\documentclass[11pt]{amsart}
\usepackage{mathrsfs}
\usepackage{amssymb}
\usepackage{hyperref,enumitem,cleveref}
\setlist[enumerate,1]{label=(\roman*),ref=(\roman*)}
\usepackage{xspace}
\hypersetup{
    colorlinks=true,
    linkcolor=purple,
    citecolor=magenta,
    filecolor=cyan,      
    urlcolor=cyan,
    pdftitle={Group Gradings on Triangularizable Algebras},
    pdfpagemode=FullScreen,
    }

\usepackage{xcolor,soul}
\newcommand{\highlight}[1]{\textit{#1}}

\newtheorem{Thm}{Theorem}
\newtheorem{Lemma}[Thm]{Lemma}
\newtheorem{Prop}[Thm]{Proposition}
\newtheorem{Cor}[Thm]{Corollary}
\theoremstyle{definition}
\newtheorem{Def}[Thm]{Definition}
\newtheorem{Rem}[Thm]{Remark}
\theoremstyle{remark}
\newtheorem{Example}{Example}
\newcommand{\F}{\ensuremath{\mathbb{F}}\xspace} 
\newcommand{\NN}{\ensuremath{\mathbb{N}}\xspace} 
\newcommand{\ZZ}{\ensuremath{\mathbb{Z}}\xspace} 
\newcommand{\RR}{\ensuremath{\mathbb{R}}\xspace} 
\newcommand{\G}{\ensuremath{G}\xspace} 
\newcommand{\Eset}{\ensuremath{\mathscr{E}}\xspace}
\newcommand{\algbr}[1]{\ensuremath{\mathcal{#1}}\xspace} 
\newcommand{\Aalg}{\algbr{A}} 
\newcommand{\Balg}{\algbr{B}} %
\newcommand{\Dalg}{\algbr{D}} %
\newcommand{\Ialg}{\algbr{I}} %
\newcommand{\Ralg}{\algbr{R}} %
\newcommand{\Ualg}{\algbr{U}} %
\newcommand{\B}{\ensuremath{\beta}\xspace} 
\newcommand{\V}{\ensuremath{\mathcal{V}}\xspace}
\newcommand{\W}{\ensuremath{\mathcal{W}}\xspace}
\newcommand{\range}{\operatorname{range}}
\newcommand{\End}{\operatorname{End}}
\newcommand{\EndF}{\End_{\F}}
\newcommand{\EndFV}{\ensuremath{\EndF\V}\xspace}
\newcommand{\ann}{\operatorname{ann}}
\newcommand{\UT}{\operatorname{UT}}
\newcommand{\UTB}{\UT_{\B}}
\newcommand{\UTm}{\UT_{\B_m}}
\newcommand{\UTBV}{\ensuremath{\UTB\V}\xspace}
\newcommand{\fin}{\operatorname{fin}}
\newcommand{\UTfin}{\UT^{\fin}}
\newcommand{\UTBlim}{\UT\limits_{\to\B}}
\newcommand{\UTBVlim}{\ensuremath{\UTBlim\V}\xspace}
\newcommand{\UTmlim}{\UT\limits_{\to\B_m}}
\newcommand{\UTBfin}{\UTfin_{\B}}
\newcommand{\UTBVfin}{\ensuremath{\UTBfin\V}\xspace}
\newcommand{\Span}{\operatorname{Span}}
\newcommand{\Supp}{\operatorname{Supp}}
\newcommand{\TNil}{\operatorname{TNil}}


\title{Group gradings on triangularizable algebras}
\author{Waldeck Sch\"utzer}
\address{Department of Mathematics, Universidade Federal de S\~{a}o Carlos, Brazil}
\email{waldeck@dm.ufscar.br}

\author{Felipe Yukihide Yasumura}
\address{Department of Mathematics, Instituto de Matem\'atica e Estat\'istica, Universidade de S\~ao Paulo, SP, Brazil}
\email{fyyasumura@ime.usp.br}
\thanks{The second named author acknowledges the support of the S\~ao Paulo Research Foundation (FAPESP), grants no.\,2018/23690-6 and no.\,2023/03922-8.}

\begin{document}
\begin{abstract}
Classifying isomorphism classes of group gradings on algebras presents a compelling challenge, particularly within the realms of non-simple and infinite-dimensional algebras, which have been relatively unexplored. This study focuses on a kind of algebra that is neither simple nor finite-dimensional, aiming to classify the group gradings on triangularizable algebras as defined by Mesyan in 2019. The topology of infinite-dimensional algebras, along with the role of idempotent elements, plays a crucial role in our findings, leading to new insights and a deeper understanding of their structure.
\end{abstract}
\maketitle


\section{Introduction}

The classification of group gradings on algebras has garnered significant interest from researchers, particularly since the foundational work of Patera and Zassenhaus \cite{PZ89}. Their systematic study primarily targeted simple Lie algebras and laid the groundwork for subsequent influential studies, such as those presented in \cite{BaSeZa2001, bashza, BaZa2002, BD2013, BaKoch, DM2006, Elduque}. The latest advancements in this field are comprehensively summarized in the monograph \cite{EK13}. Notably, there were classification efforts preceding Patera and Zassenhaus's work, including \cite{Zelm}. Recently, significant progress was made with the classification of group gradings on finite-dimensional associative, Lie, and Leibniz algebras via their quantum symmetric semigroups \cite{AgoMil, Mil}.

While most research efforts have focused on finite-dimensional algebras, there exist numerous significant infinite-dimensional graded algebras, such as polynomial algebras, graded algebras derived from filtrations, and Leavitt path algebras (see, for instance, \cite{Hazbook}). Classifying group gradings on a given infinite-dimensional algebra is far from trivial, and little is known in this area.

One of the pioneering studies addressing this issue is \cite{BaZa2010}, which classified gradings on infinite matrices with finitely many nonzero entries over an algebraically closed field of characteristic zero. Additionally, by employing Functional Identities techniques, researchers have elucidated the structure of specific group gradings on Lie and Jordan algebras \cite{BaBr, BaBrSh}. Further extending this research, \cite{BaBrKoch} provided a classification of group gradings on infinite-dimensional primitive algebras with minimal ideals, thereby classifying abelian group gradings on infinite-dimensional finitary simple Lie algebras.

All the mentioned works deal with algebras closely related to simple ones. In another direction, some papers focus on classifying group gradings on non-simple algebras. Bahturin classified the group gradings on a free nilpotent Lie algebra \cite{Bahturin}. For finite-dimensional associative algebras, the isomorphism classes of group gradings on the upper triangular matrices $\UT_n\F$ over an arbitrary field \F were classified in \cite{VinKoVa2004,VZ2003,VZ2007}, showing that every group grading on $\UT_n\F$ is elementary. Moreover, a classification of the isomorphism classes of group gradings was provided. Next, a classification of group gradings on upper-block triangular matrices was presented in \cite{BFD2018,VZ2011,KochY,Y}. The works \cite{BesDas,Jon2006,MSp2010,Pr2013,SSY} investigated group gradings on incidence algebras. Following this trend, the next natural step would be to consider upper triangular matrices in the infinite-dimensional setting. However, there is no unique way to define an infinite upper triangular matrix. One approach, as done in \cite{BaZa2010}, is to consider the direct limit $\lim\limits_\to\mathrm{UT}_n$. A much more interesting possibility has been recently opened by Mesyan \cite{Mes}, allowing for the construction of an algebra of infinite upper triangular matrices that can have infinitely many nonzero columns, even uncountably many, yet retain only finitely many nonzero entries in each column. See also \cite{EIS,Mesb}.

Following these works, let \V be an infinite-dimensional vector space over a field \F, and let $(\B, \leq)$ be a well-ordered basis of \V. For each $v \in \B$, let $\V(v) = \Span\{w \in \B \mid w \leq v\}$. Mesyan defines an element $t \in \EndFV$ as being upper triangular with respect to $(\B, \leq)$ if $t(v) \in \V(v)$ for all $v \in \B$. Here, we define \UTBV as the subset of all upper triangular elements in \EndFV with respect to the basis \B (cf. \cite[Definition 4]{Mes}). This is a triangularizable algebra in the sense of \cite[Definition 2.1]{Mesb}.

For each pair of nonzero vectors $u, v \in \B$, we let $e_{uv}$ be the unique transformation in \EndFV defined by
\[
e_{uv}(w) = \begin{cases}
u, & \text{if } w = v, \\
0, & \text{otherwise}.
\end{cases}
\]
for all $w \in \B$. These are called matrix units relative to \B. In particular, $e_{vv}$ is the projection of \V onto the subspace spanned by $v$, with kernel $\Span(\B \setminus \{v\})$. Every $t \in \EndFV$ satisfies $e_{uu} t e_{vv} = a_{uv} e_{uv}$ for some $a_{uv} \in \F$. We find it useful and convenient to define $t(u, v) = a_{uv}$, so that $e_{uu} t e_{vv} = t(u, v) e_{uv}$ for all $u, v \in \B$. It should be noted that this notation is ambiguous as it bears no explicit reference but depends on the basis \B.

Other examples of triangularizable algebras of interest include $\UTBVfin = \{t \in \UTBV \mid \dim t(\V) < \infty\}$ and $\UTBVlim = \Span\{e_{uv} \mid u \leq v \in \B\}$. The former can be identified with $\UT_n \F$ when $|\B| = n$ is finite, in which case $\UTBV = \UTBVfin = \UTBVlim$. The latter can be identified with the direct limit $\lim\limits_{\to} \UT_n \F$ when $\B = \{v_n \mid n \in \NN\}$ is countable.
The algebra \UTBV is clearly a unital subalgebra of \EndFV, and both \UTBVfin and \UTBVlim are generally non-unital subalgebras of \UTBV. Furthermore, from \cite[Lemma 2.5]{Mesb}, the closure of \UTBV in the function topology of \EndFV is still triangular with respect to \B; hence, \UTBV is topologically closed in \EndFV.

Recall that an \F-algebra \Aalg is graded by a group \G if there exist subspaces $\Aalg_g$, for $g \in \G$, satisfying $\Aalg = \bigoplus_{g \in \G} \Aalg_g$ and $\Aalg_g \Aalg_h \subseteq \Aalg_{gh}$, for all $g, h \in \G$. In direct analogy with the finite-dimensional case, we consider good and elementary gradings as follows:
\begin{Def}
A \G-grading on \Aalg, where $\UTBVlim \subseteq \Aalg \subseteq \UTBV$, is a \highlight{good grading} if every matrix unit $e_{uv}$ ($u, v \in \B$) is homogeneous in the grading. The grading is \highlight{elementary} if there is a function $\gamma \colon \B \to \G$ such that $e_{uv}$ is homogeneous of degree $\gamma(u)^{-1}\gamma(v)$.
\end{Def}
We notice that if each $e_{uv}$ is homogeneous and the grading is finite, then the grading is completely determined (see \Cref{inducedfromgood}). This fact is far from obvious in the present context. In fact, it is not at all obvious why \Aalg should have any nontrivial gradings.

Here, we prove that every group grading on \UTBVlim, \UTBVfin and \UTBV is isomorphic to an elementary grading. Moreover, the group gradings on the latter two are finite and induced from a finite group grading on the former. We also classify the isomorphism classes of group gradings on these algebras (see \Cref{statementmainthm}).

Finally, we remark that if \B is uncountably infinite, then \UTBV is isomorphic to the finitary incidence algebra of the non locally-finite partially ordered set $(\B,\leq)$ (cf. \cite{Khrip}), and if \B is finite or \NN, then \UTBV is isomorphic to the incidence algebra of the locally finite partially ordered set $(\B,\leq)$. 

\section{Preliminaries}\label{sec:prelim}

\subsection{Topological algebras}\label{subsec:top:alg}

A \highlight{topological vector space} is a vector space \V endowed with a topology such that the sum $\V \times \V \to \V$ and the scalar multiplication $\F \times \V \to \V$ are continuous maps. The Cartesian products have the product topology, and \F has the Hausdorff topology. In general, we shall assume that \F is endowed with the discrete topology. A \highlight{topological algebra} is a topological vector space \Aalg endowed with a continuous bilinear map $\Aalg \times \Aalg \to \Aalg$. All the topological algebras in this paper will be Hausdorff.

Let $X$ be a topological space and $Y$ a set. Then, the space $X^Y=\{f\colon Y\to X\text{ map}\}$ has a natural topology given by the product topology if we identify $X^Y$ with $\prod_YX$. This topology on $X^Y$ is called the \highlight{topology of pointwise convergence}. In the particular case where $X$ has the discrete topology, the topology on $X^Y$ is called the \highlight{finite topology}. In this case, a basis of neighborhoods of $f\in X^Y$ is given by
$$
\mathscr{N}(f,S)=\{g\in X^Y\mid g(x)=f(x),\forall x\in S\},
$$
where $S\subseteq Y$ is a finite set.

Now, let \V be a vector space endowed with the discrete topology, so that $\V^{\V}$ is endowed with the Hausdorff finite topology. 
Then, the closed subset $\EndFV\subseteq\V^\V$ inherits the finite topology, and thus it is a Hausdorff topological algebra. A basis of neighborhoods of $0$, in this case, is given by
$$
\mathscr{N}(\W)=\{T\in\EndFV\mid T|_\W=0\},
$$
where $\W\subseteq\V$ is finite-dimensional. Indeed, let $S$ be a finite set that spans \W. Then $\mathscr{N}(\W)=\mathscr{N}(0,S)$, in the above notation.

Let $\Eset=\{e_i\mid i\in J\}\subseteq\V$, where \V is a Hausdorff topological vector space,  so that every convergent net has a unique limit in \V. We say that \Eset is \highlight{summable} if the net
$$
\left\{\sum_{i\in S}e_i\mid S\subseteq J\text{ is finite}\right\}
$$
converges. In this case, it is usual to denote the limit by $\sum_{i\in J}e_i$ or by $\sum_{e\in\Eset}e$.

Let $(J,\leq)$ be a directed set. Recall that a \highlight{net} in a topological space $X$ is a function from $J$ to $X$. A net $\{x_j\}$ is said to \highlight{converge} to $x\in X$ if for every neighborhood $U$ of $x$, there exists $j_0\in J$ such that $x_j\in U$ for all $j\geq j_0$. A net $(x_\alpha)_{\alpha\in J}$ in a topological vector space \V is called a \highlight{Cauchy net} if for any neighborhood $U$ of $0$, there exists $\gamma\in J$ such that $\alpha,\alpha'\ge\gamma$ implies $x_\alpha-x_{\alpha'}\in U$. A topological vector space is called \highlight{complete} if every Cauchy net converges. Moreover, if \V is complete and $\W \subseteq \V$ is a closed subspace, then \W is complete as well. Since \V is Hausdorff and complete, so are $\V^{\V}$ and \EndFV.

The following extension result will be very useful for our purposes. It is a particular version of \cite[Theorem 2, \S3.6, Chapter 2]{Bourbaki}.
\begin{Lemma}\label{ext_cont_linmap}
Let \V, \W be Hausdorff topological vector spaces, where \W is complete, and let $\V_0\subseteq\V$ be a dense subspace. If $f\colon\V_0\to\W$ is linear and continuous, then there exists a unique extension $\bar{f}\colon\V\to\W$ which is linear and continuous.
\end{Lemma}

\subsection{Graded algebras}\label{sec:graded:algebras}

Let \Aalg be an arbitrary algebra over a field \F, and let \G be any group. We use multiplicative notation for \G, and denote its identity by $1$. We say that \Aalg is \highlight{\G-graded} if \Aalg is endowed with a fixed vector space decomposition,
\[
\Gamma\colon\Aalg=\bigoplus_{g\in \G}\Aalg_g,
\]
where some of the subspaces $\Aalg_g$ may be $0$, and such that $\Aalg_g\Aalg_h\subseteq \Aalg_{gh}$, for all $g, h\in \G$. The subspace $\Aalg_g$ is called the \highlight{homogeneous component of degree $g$}, and the non-zero elements $x\in\Aalg_g$ are said to be \highlight{homogeneous} of degree $g$. We write $\deg x=g$ for these elements. We also write $\deg_\G$ or $\deg_\Gamma$ to make it clear that we consider the degree with respect to the grading $\Gamma$. Not all elements of \G are required to participate in the grading with corresponding nonzero homogeneous subspace, hence it is often useful to define the \highlight{support of the grading}, $\Supp\Gamma=\{g\in \G\mid\Aalg_g\ne0\}$, to single out those elements which do participate. A \G-grading is called \highlight{finite} if its support is finite.

A subspace $\mathcal{S}\subseteq\Aalg$ is \highlight{homogeneous} of degree $g$ if $\mathcal{S}\subseteq\Aalg_g$.
For an arbitrary subspace $\mathcal{S}\subseteq\Aalg$, the subspace $\mathcal{S}_g=\mathcal{S}\cap\Aalg_g$
is the (possibly zero) \highlight{homogeneous component} of $\mathcal{S}$ of degree $g$. We say that a subspace $\mathcal{S}$ is \highlight{graded} if it is the sum of all its homogeneous components, namely $\mathcal{S}=\bigoplus_{g\in \G}(\mathcal{S}\cap\Aalg_g)$.

Let $\Gamma'\colon\Balg=\bigoplus_{g\in \G}\Balg_g$ be another \G-graded algebra. A map $f\colon\Aalg\to\Balg$ is called a \highlight{homomorphism of \G-graded algebras} if $f$ is a homomorphism of algebras and $f(\Aalg_g)\subseteq\Balg_g$, for all $g\in \G$. If, moreover, $f$ is an isomorphism, we call $f$ a \highlight{\G-graded isomorphism} (or an isomorphism of \G-graded algebras), and we say that $(\Aalg,\Gamma)$ and $(\Balg,\Gamma')$ are \highlight{\G-graded isomorphic} (or isomorphic as \G-graded algebras). Two \G-gradings, $\Gamma$ and $\Gamma'$, on the same algebra \Aalg are \highlight{isomorphic} if $(\Aalg,\Gamma)$ and $(\Aalg,\Gamma')$ are isomorphic as graded algebras. In this case, we write $\Gamma\cong\Gamma'$.

If $f\colon\Aalg\to\Aalg$ is an algebra automorphism, then $f$ induces a new \G-grading on \Aalg via $\Aalg=\bigoplus_{g\in \G}\Aalg'_g$, where $\Aalg'_g=f(\Aalg_g)$. Note that $f$ will be an isomorphism between these two gradings on \Aalg.

\section{Idempotents and their immediate properties}

In this section, we introduce the notion of \B-support and use it to derive key properties of idempotents and their associated algebraic structures, highlighting their importance and utility in our study.

If \B is a basis of \V, we define the \highlight{\B-support} of $t\in\EndFV$ as the set
\[
\Supp_{\B}(t)=\{v\in\B\mid e_{vv} t e_{vv}\neq 0\}.
\]
If $\B=\{v_1,v_2,\ldots,v_n\}$ is finite, this is the subset of indices corresponding to the nonzero diagonal entries of the matrix $[t]_{\B}$ of $t$. When there is no risk of confusion, for instance when \B is clear from the context, we shall simply write $\Supp(t)$ for $\Supp_{\B}(t)$ and call it the support of $t$. We
obviously have $v\in\Supp(t)$ if, and only if, $t(v,v)\neq 0$.

\begin{Rem}
When $s$ and $t$ are triangular with respect to the well-ordered basis $(\B,\leq)$, it is straightforward to show that $e_{vv}ste_{vv}=(e_{vv}s e_{vv})(e_{vv}te_{vv})$. Hence, $(st)(v,v)=s(v,v)t(v,v)$ for all $v\in\B$, and thus $\Supp(st)=\Supp(s)\cap\Supp(t)$.
\end{Rem}

The next lemma allows us to construct idempotents with desired properties, which will be very useful later on.

\begin{Lemma}\label{JCD}
Let $t\in\EndFV$ be triangularizable and such that $\dim t(\V)<+\infty$. Then there exist polynomials  $p(\lambda),f(\lambda)\in\F[\lambda]$ without constant term such that $p(t)=0$ and $e=f(t)$ is an idempotent with $\Supp(e)=\Supp(t)$.
Consequently, $t$ is algebraic over \F, $e$ commutes with every endomorphism commuting with $t$ and stabilizes every subspace of \V stabilized by $t$.
\end{Lemma}
\begin{proof}
By \cite[Proposition 6]{Mes}, since $t$ is triangularizable and $\dim t(\V)<+\infty$, there exists a finite dimensional $t$-invariant subspace $\W \subset \V$ such that $t(\V) \subset \W$. Further, by \cite[Theorem 8(2)]{Mes}, there is a polynomial $p_0(\lambda)\in\F[\lambda]$ that factors into linear terms and such that $p_0(t)$ annihilates \W. It follows that $p(\lambda)=p_0(\lambda)\lambda$ is such that $p(t)$ annihilates \V.
We may assume that $p(\lambda)=\prod_{j=1}^{k}(\lambda-\alpha_j)^{m_j}$ for distinct scalars $\alpha_1,\ldots,\alpha_{k}$, with $\alpha_k=0$, and natural numbers $m_1,\ldots,m_k$. Then, by \cite[Lemma 2]{Mes}, $\V=\bigoplus_{j=1}^k\ker((t-\alpha_j I)^{m_j})$. 

By the Chinese Remainder Theorem, there exists a polynomial $f(\lambda)\in\F[\lambda]$ satisfying the congruences
\begin{eqnarray*}
f(\lambda) &\equiv& \delta_j\mod((\lambda-\alpha_j)^{m_j})\quad (1\leq j\leq k),\\
f(\lambda) &\equiv& 0\mod (\lambda),
\end{eqnarray*}
\[\begin{cases}
f(\lambda) \equiv 1\mod((\lambda-\alpha_j)^{m_j}),& 1\leq j< k,\\
f(\lambda) \equiv 0\mod (\lambda^{m_k}),& j=k.
\end{cases}\]
Let $e=f(t)$. In view of the last congruence, $f$ has no constant term, hence the last assertions about $e$ are obviously true. The $j$-th congruence means that $e$ acts diagonally on the subspace $\ker((t-\alpha_j I)^{m_j})$ of \V as a scalar $\delta_j\in\{0,1\}$, hence this element is an idempotent. Lastly, as $e$ stabilizes every subspace stabilized by $t$, of course $\Supp(e)=\Supp(t)$, and this completes the proof.
\end{proof}

The next lemma generalizes well-known properties of idempotents in the finite dimensional algebra $\UT_n\F$. Recall that two idempotents $f,g$ in a ring $R$ are isomorphic, written $f\cong g$, if $fR\cong gR$ as right $R$-modules.

\begin{Lemma}\label{lem:iso:idemps}
Let $f,g\in\EndFV$ be idempotents
which are triangular with respect to a common well-ordered basis $(\B,\leq)$. The following are equivalent:
\begin{enumerate}
    \item\label{lem:iso:idemps:1} $\Supp_\B(f)=\Supp_\B(g)$;
    \item\label{lem:iso:idemps:4} There is an invertible element $h\in\EndFV$ also triangular with respect to $(\B,\leq)$ and such that $f=hgh^{-1}$.
    \item\label{lem:iso:idemps:5}There exist $a,b\in\EndFV$, with $a,b$ triangular with respect to \B, such that $f=ab$ and $g=ba$.
    \item\label{lem:iso:idemps:6} $f\cong g$ in \EndFV. 
\end{enumerate}
Further, if $f$ and $g$ commute and have any of the above equivalent properties, then $f=g$.
\end{Lemma}
\begin{proof}
\labelcref{lem:iso:idemps:1}$\implies$\labelcref{lem:iso:idemps:4}: Define $h=f g+(1-f)(1-g)$. This transformation is still triangular with respect to $(\B,\leq)$ and satisfies $fh=fg=hg$. Now, each $v\in\B$ is either in $\Supp_\B(f)=\Supp_\B(g)$ or in $\Supp_\B(1-f)=\Supp_\B(1-g)$, hence $h(v,v)\neq 0$ for all $v\in\B$. It follows from \cite[Proposition 16]{Mes} that $h$ is invertible and that $h^{-1}$ is also triangular with respect to $(\B,\leq)$. 

\labelcref{lem:iso:idemps:4}$\implies$\labelcref{lem:iso:idemps:5}: Take $a=h$ and $b=gh^{-1}$.

\labelcref{lem:iso:idemps:5}$\implies$\labelcref{lem:iso:idemps:1}: Obvious, since $\Supp_{\B}(ab)=\Supp_{\B}(ba)$.

\labelcref{lem:iso:idemps:5}$\iff$\labelcref{lem:iso:idemps:6}: This is \cite[Proposition 21.20]{Lam}.

The last claim is \cite[Exercise 21.21]{Lam}.
\end{proof}

If $f$ is an idempotent in a triangularizable algebra $\Ralg$, then of course the corner subalgebra $f\Ralg f$ is still a triangularizable algebra with respect to the same well-ordered basis as $\Ralg$.
The next proposition makes this observation more precise.

\begin{Prop}\label{prop:triang:corner}
Let $\Ralg$ be a triangularizable subalgebra of \EndFV with respect to the well-ordered basis $(\B,\leq)$, $\Aalg\subseteq\Ralg$ a subalgebra, and $f\in\Ralg$ an idempotent. Also, let $\B'=\Supp_{\B}(f)$ and $\V'=\Span\B'$. Then, the corner
subalgebra $\Dalg=f \Aalg f$ of \Aalg is isomorphic as a topological \F-algebra to a triangularizable subalgebra $\Dalg'$ of $\EndF\V'$. Moreover, if \Aalg is the set of all transformations in \EndFV which are triangular with respect to $(\B,\leq)$, then so is $\Dalg'$ in $\EndF\V'$ with respect to $(\B',\leq)$.
\end{Prop}
\begin{proof}
Let $g$ be the unique idempotent element of \EndFV such that $g(v)=v$ for all $v\in\B'$ and with kernel the $\Span(\B\setminus\B')$. Since $\Supp_{\B}(f)=\Supp_{\B}(g)$, \Cref{lem:iso:idemps} provides an invertible element $h\in\EndFV$ such that $g=h f h^{-1}$, with both $h,h^{-1}$ still triangular with respect to $(\B,\leq)$. Let $\varphi\colon\EndFV\to\EndFV$ be the continuous inner automorphism such that $\varphi(t)=hth^{-1}$.

For every $t\in \Dalg$, we have $t=ft=tf=ftf$, hence $\varphi(t)(\V)=\varphi(ft)(\V)=(g\varphi(t))(\V)\subseteq\V'$, while
$\varphi(t)(v)=\varphi(tf)(v)=(\varphi(t)g)(v)=0$ for all $v\in\B\setminus \B'$, hence $\varphi(t)\in\EndF\V'$ and this transformation is triangular with respect to $(\B,\leq)$. Here we identify $\EndF\V'$ with the subset of transformations in \EndFV which keep $\V'$ invariant and vanish on $\Span(\B\setminus\B')$. Under this identification, the element $g$ becomes the identity transformation of $\V'$. It follows that $\Dalg'=\varphi(\Dalg)$ is a triangularizable subalgebra of \EndFV with respect to $(\B,\leq)$ as well as a triangularizable subalgebra of $\EndF\V'$ with respect to $(\B',\leq)$.

Now assume that \Aalg is the set of all transformations which are triangular with respect to $(\B,\leq)$ (in which case $\Ralg=\Aalg$) and let $t\in \EndF\V'$ be triangular with respect to $(\B',\leq)$. Since $t(v)\in\Span\{u\in\B'\mid u\leq v\}$ for all $v\in\B'$ and $t(v)=0$ for all $v\in\B\setminus\B'$, it is clear that the element $s=\varphi^{-1}(t)=h^{-1} t h$ satisfies $s(v)\in\Span\{u\in\B\mid u\leq v\}$ for all $v\in\B$, so it is triangular with respect to $(\B,\leq)$, and thus it lies in \Aalg. Also $sf=h^{-1}thfh^{-1}h=h^{-1}tg h=h^{-1}th=s$, because $g$ is the identity in $\EndF\V'$, and similarly $fs=s$, hence $s\in \Dalg$. But then $t=\varphi(s)\in \Dalg'$, as required.  
\end{proof}

In view of the previous proposition, it follows that:

\begin{Cor}\label{cor:corner:UTBV}
Let $f\in\UTBV$ be an idempotent, $\B'=\Supp_{\B}(f)$ and $V'=\Span{\B'}$. Then $f(\UTBV)f\cong \UT_{\B'}V'$ as topological \F-algebras.
In particular,
if $|\Supp_{\B}(f)|=n<+\infty$, then $f(\UTBV)f\cong\UT_n\F$.\qed
\end{Cor}

Note that, if we consider an algebra \Aalg, where $\UTBVlim\subseteq\Aalg\subseteq\UTBV$, then an indempotent $f\in\Aalg$ is an element of \UTBV. Hence, we may repeat the proof of \Cref{prop:triang:corner}, and obtain:

\begin{Cor}\label{cor:corner:A}
Let \Aalg be a subalgebra of \UTBVfin containing \UTBVlim,
and let $f\in\Aalg$ be an idempotent. Then $f\Aalg f\cong\Aalg'$, where $\UT_{\to\B'}\V\subseteq\Aalg'\subseteq\UT_{\B'}\V'$ as topological \F-algebras,
and where $\B'=\Supp_{\B}(f)$ and $\V'=\Span{\B'}$. In particular,
if $|\Supp_{\B}(f)|=n<+\infty$, then $f\Aalg f\cong\UT_n\F$.\qed
\end{Cor}

\section{Gradings on topological algebras}

Let \Aalg be a topological algebra. This means that \Aalg is a topological Hausdorff vector space endowed with a continuous bilinear product $(a,b)\mapsto ab$, where $\Aalg\times\Aalg$ is endowed with the product topology.

\begin{Def}
Let $\Gamma\colon\Aalg=\bigoplus_{g\in \G}\Aalg_g$ be a \G-grading on the topological algebra \Aalg, and denote by $\pi_g\colon\Aalg\to\Aalg$ the respective projection of \Aalg onto $\Aalg_g$.
\begin{enumerate}
\item We say that $\Gamma$ is \highlight{closed} if each $\Aalg_g$ is a topologically closed subspace.
\item We say that $\Gamma$ is \highlight{continuous} if each $\pi_g$ is a continuous map.
\end{enumerate}
\end{Def}

When \Aalg is finite dimensional, it is natural to consider the discrete topology, in which case every \G-grading is obviously continuous and closed. The next lemma shows that every continuous grading is closed; however, the converse need not hold. See \Cref{exem:closed:not:continuous}.

\begin{Lemma}\label{continuousgradingisclosed}
A continuous \G-grading is a closed one.
\end{Lemma}
\begin{proof}
Indeed, since the topology on \Aalg is Hausdorff, $\{0\}$ is closed. So $\pi_g^{-1}(0)=\bigoplus_{h\ne g}\Aalg_h$ is also closed. Thus, $\Aalg_g=\bigcap_{h\ne g}\pi_h^{-1}(0)$ is closed as well.
\end{proof}

Let $\Aalg_0\subseteq\Aalg$ be a dense subalgebra of a complete topological algebra, such that $\Aalg_0$ is a subalgebra and $\bar{\Aalg}_0=\Aalg$ (the topological closure). If $\Gamma$ is a continuous \G-grading on $\Aalg_0$, then the projections $\pi_g\colon\Aalg_0\to\Aalg_0$ are continuous. Hence, from \Cref{ext_cont_linmap}, there exists a unique continuous extension $\bar{\pi}_g\colon\Aalg\to\Aalg$.

\begin{Lemma}\label{extensioncontinuous}
In the above notation, if the grading is finite, then 
$\{\bar{\pi}_g\mid g\in \G\}$ is a complete set of projections of \Aalg. Moreover, $\Aalg=\bigoplus_{g\in \G}\bar{\pi}_g\left(\Aalg\right)$ is a continuous \G-grading on \Aalg, and $\bar{\pi}_g(\Aalg)=\overline{\pi_g(\Aalg_0)}$.
\end{Lemma}
\begin{proof}
Let $x\in\Aalg$ and $(x_\alpha)_{\alpha\in J}$ be a net on $\Aalg_0$ converging to $x$. Then, by definition, $\bar{\pi}_g(x)=\lim\pi_g(x_\alpha)$. First, from the continuity of $\bar{\pi}_g$, we have
$$
\bar{\pi}^2_g(x)=\lim\pi_g^2(x_\alpha)=\lim\pi_g(x_\alpha)=\bar{\pi}_g(x),
$$
so, $\bar{\pi}_g^2=\bar{\pi}_g$, for each $g\in \G$. Also, whenever $g\neq h$, we have
$$
\bar{\pi}_g\bar{\pi}_h(x)=\lim\pi_g\pi_h(x_\alpha)=\lim0=0,
$$
hence $\bar{\pi}_g\bar{\pi}_h=0$. Finally, since the grading is finite,
$$
\sum_{g\in \G}\bar{\pi}_g(x)=\lim\sum_{g\in \G}{\pi}_g(x_\alpha)=\lim x_\alpha=x,
$$
thus, $\sum_{g\in \G}\bar{\pi}_g=1$. Hence, $\{\bar{\pi}_g\mid g\in \G\}$ is a complete set of projections of \Aalg. Let us show that the corresponding vector space decomposition, namely $\Aalg=\bigoplus_{g\in \G}\bar{\pi}_g(\Aalg)$, is a \G-grading on \Aalg. Let $g,h\in \G$, $x,y\in\Aalg$, and $\{(x_\alpha,y_\alpha)\}_{\alpha\in J}$ be a net on $\Aalg_0\times\Aalg_0$ converging to $(x,y)$ (with respect to the product topology). We have,
\begin{align*}
\bar{\pi}_g(x)\bar{\pi}_h(y)&=\pi_g(\lim x_\alpha)\pi_h(\lim y_\alpha)=(\lim\pi_g(x_\alpha))(\lim\pi_h(y_\alpha))\\
&=\lim\pi_g(x_\alpha)\pi_h(y_\alpha)=\lim\pi_{gh}(x_\alpha y_\alpha)=\bar{\pi}_{gh}(\lim x_\alpha\cdot y_\alpha)\\&=\bar{\pi}_{gh}(xy).
\end{align*}

By construction, each $\bar{\pi}_g$ is a continuous map, so the \G-grading is a continuous one. By construction, $\bar{\pi}_g(\Aalg)$ is obtained from limits of nets in $\pi_g(\Aalg_0)$, so $\bar{\pi}_g(\Aalg)\subseteq\overline{\pi_g(\Aalg_0)}$. On the other hand, clearly $\pi_g(\Aalg_0)\subseteq\bar{\pi}_g(\Aalg)$. Since the grading defined by the $\{\bar{\pi}_g\}$ is continuous, it is also closed (\Cref{continuousgradingisclosed}). Thus, $\overline{\pi_g(\Aalg_0)}\subset\bar{\pi}_g(\Aalg)$, so equality holds.
\end{proof}

\begin{Rem}
Another way to prove that $\Aalg=\bigoplus_{g\in \G}\bar{\pi}_g(\Aalg)$  defines a grading is  as follows. First, let $f_1,\dots,f_m\colon\Aalg_0\to\Aalg_0$ be continuous linear maps, $\bar{f}_1,\dots,\bar{f}_m\colon\Aalg\to\Aalg$ their respective linear extensions, and $p=p(x_1,\ldots,x_m)\in\F\langle x_1,\ldots,x_m\rangle$ be such that $p(f_1,\ldots,f_m)=0$, where $\F\langle x_1,\ldots,x_m\rangle$ is the free associative algebra of rank $m$. Given $x\in\Aalg$, let $(x_\alpha)_{\alpha\in J}$ be a net in $\Aalg_0$ converging to $x$, then
$$
p(\bar{f}_1,\ldots,\bar{f}_m)x=\lim_\alpha p(f_1,\ldots,f_m)x_\alpha=0.
$$
Hence, $p(\bar{f}_1,\ldots,\bar{f}_m)=0$. Now, assume that $\Aalg_0=\bigoplus_{i=1}^m\pi_{g_i}(\Aalg_0)$ is a finite \G-grading on $\Aalg_0$.
Then, $\{\pi_{g_i}\mid i=1,\ldots,m\}$ satisfies the polynomials
$$
p_{ij}(x_i,x_j)=x_ix_j-\delta_{ij}x_i,\quad p(x_1,\ldots,x_m)=x_1+x_2+\cdots+x_m-1.
$$
Thus, the set of extensions $\{\bar{\pi}_{g_i}\mid i=1,\ldots,m\}$ satisfies the same polynomials, so it constitutes a (vector space) \G-grading on \Aalg.
\end{Rem}

As mentioned above, \EndFV is a complete topological space and $\UTBV\subseteq\EndFV$ is a closed subalgebra. Thus, \UTBV is complete as well. Hence, every finite continuous grading on \UTBVlim extends to a continuous grading on \UTBV. The next example shows that we cannot always do without the finiteness condition.

\begin{Example}
Let $\B=\{v_i\mid i\in\NN\}$ and, for simplicity, denote $e_{v_iv_j}$ just by $e_{ij}$. Let $\G=\ZZ$ be the free abelian group with 1 generator, and let $\gamma\colon\B\to\ZZ$ be defined by 
$$
\gamma(v_i)=\displaystyle\frac{i(i+1)}2.
$$
Then, $\gamma$ defines a good \ZZ-grading on \UTBVlim if we set $\deg e_{1i}=\gamma(v_i)$ (see \Cref{sec_elem_grad}), which is equivalent to setting $\deg e_{i,i+1}=i+1,\,\forall i\in\NN$.
Denote, as before, $\pi_i\colon\UTBVlim\to\UTBVlim$ the respective projections, and let $\bar{\pi}_i\colon\UTBV\to\UTBV$ be the continuous extensions. Then $\{\bar{\pi}_i\mid i\in\NN\}$ is a set of pairwise orthogonal idempotents. However, it is not true that their image sums to the whole space \UTBV. Indeed, let $a=\sum_{i\in\NN}e_{i,i+1}$. Then,
$$
\bar{\pi}_j(a)=\lim_i\pi_j(\sum_{k=1}^ie_{k,k+1}).
$$
The term inside the limit equals $e_{j,j+1}$ for $i\ge j$ and $0$ otherwise. Hence, $\bar{\pi}_j(a)=e_{j,j+1}\ne0$, for all $j\in\NN$. Thus,
$$
a\notin\bigoplus_{j\in\NN}\bar{\pi}_j(\UTBV).
$$
Therefore, the grading defined by $\gamma$ on \UTBVlim does not admit an extension to a \ZZ-grading on \UTBV.
\end{Example}

\noindent\textbf{Notation.} Let $\Aalg_0\subseteq\Aalg$ be a dense subalgebra, where \Aalg is complete. If $\Gamma$ is a continuous \G-grading on $\Aalg_0$, we denote by $\overline{\Gamma}$ its extension to \Aalg.

\begin{Rem}\label{cont_iso}
If $\Gamma\cong\Gamma'$ by a continuous isomorphism, then $\overline{\Gamma}\cong\overline{\Gamma'}$.
\end{Rem}

Now, let us study a special kind of group grading. Let \V be a vector space with a basis \B. We equip \EndFV with the finite topology. We let $\Aalg=\Aalg(I)\subseteq\EndFV$ be a subalgebra such that there exists $I\subseteq\B\times\B$ satisfying:
\begin{enumerate}
\item $(v,v)\in I$, for all $v\in\B$,
\item $\Aalg(I)=\Span\{e_{uv}\mid(u,v)\in I\}$.
\end{enumerate}
Note that, since $\Aalg(I)$ is an algebra, the relation defined by $I$ is in fact a preorder on \B. For instance, if \V is finite-dimensional, then examples of such algebra include the full matrix algebras, the upper triangular ones, and more generally the incidence algebras of a finite set equipped with a pre-ordering. If \B is well-ordered and $I=\{(x,y)\in \B\times \B\mid x\le y\}$, then we obtain \UTBVlim.

The algebra $\Aalg=\Aalg(I)$ is known in the literature as a \highlight{structural matrix algebra} when $I$ is finite, and the notion of a good grading applies to it as well:
\begin{Def}\label{def:good:grading:A(I)}
A \G-grading on \Aalg is called \highlight{good} if every $e_{xy}$, with $(x,y)\in I$, is homogeneous in the grading.
\end{Def}

First, we shall prove that every good grading is a continuous.
\begin{Prop}\label{goodiscontinuous}
Let \G be a group and $\Gamma$ a good \G-grading on \Aalg. Then $\Gamma$ is continuous.
\end{Prop}
\begin{proof}
Consider $\Gamma\colon\Aalg=\bigoplus_{g\in \G}\Aalg_g$ and let $\pi_g$ be the projections. Let $\{r_\alpha\}_{\alpha\in\mathcal{I}}$ be a net on \Aalg converging to $r\in\Aalg$. Let $S=\{x_1,\ldots,x_m\}\subseteq\B$. Write
$$
I=I_g\dot\cup I_{\ne g}\dot\cup I_r,
$$
where
$$
I_g=\{(x,y)\in I\mid e_{xy}\in\Aalg_g,\,y\in S\},\quad I_{\ne g}=\{(x,y)\in I\setminus I_g\mid y\in S\},
$$
and the remaining $I_r=I\setminus(I_g\cup I_{\ne g})$. Then, if
$$
r=\sum_{(x,y)\in I_g}\lambda_{xy}e_{xy}+\sum_{(x,y)\in I_{\ne g}}\lambda_{xy}e_{xy}+\sum_{(x,y)\in I_r}\lambda_{xy}e_{xy},
$$
then
$$
\pi_g(r)=\sum_{(x,y)\in I_g}\lambda_{xy}e_{xy}+\sum_{(x,y)\in I_r}\lambda'_{xy}e_{xy}.
$$
Denote $\mathscr{N}(r,S)=\{z\in\EndFV\mid z(x)=r(x),\forall x\in S\}$ at the elements forming a basis of the topology of \EndFV. Then, note that
$$
z\in\mathscr{N}(r,S)\cap\Aalg\iff z\equiv \sum_{(x,y)\in I_g}\lambda_{xy}e_{xy}+\sum_{(x,y)\in I_{\ne g}}\lambda_{xy}e_{xy}\pmod{\mathrm{Span}\,I_r},
$$
and
$$
z\in\mathscr{N}(\pi_g(r),S)\cap\Aalg\iff z\equiv \sum_{(x,y)\in I_g}\lambda_{xy}e_{xy}\pmod{\mathrm{Span}\,I_r}.
$$
Hence, $\pi_g(\mathscr{N}(r,S))\subseteq\mathscr{N}(\pi_g(r),S)$. Thus, $\pi_g$ is continuous.
\end{proof}

Now, let $\Ualg$ be an algebra, where $\Aalg(I)\subseteq\Ualg\subseteq\overline{\Aalg(I)}$, where $\overline{\Aalg(I)}$ is the topological closure of $\Aalg(I)$ in \EndFV.
\begin{Def}\label{def:good:grading:U}
A \G-grading on $\Ualg$ is called \highlight{good} if every $e_{xy}$ is homogeneous in the grading, for all $(x,y)\in I$.
\end{Def}
Let $\Ualg=\bigoplus_{g\in \G}\Ualg_g$ be a finite good \G-grading. Then $\Aalg_g:=\Aalg\cap\Ualg_g$ defines a good \G-grading on \Aalg. This extends to a \G-grading on $\overline{\Aalg}$ via $\bigoplus_{g\in \G}\overline{\Aalg\cap\Ualg_g}$. It is not entirely obvious that $\overline{\Aalg_g}\cap\Ualg=\Ualg_g$. However, as the next \namecref{prop:ualg:closed} proves, every $\Ualg_g$ is a closed subspace. 
\begin{Prop}\label{prop:ualg:closed}
A good \G-grading on $\Ualg$ is closed.
\end{Prop}
\begin{proof}
Consider the good \G-grading $\Ualg=\bigoplus_{g\in \G}\Ualg_g$. First, note that if $x\in\Ualg$ is homogeneous and $e_{ii}xe_{jj}\ne0$, then $\deg_\G x=\deg e_{ij}$. So, let $g\in \G$ and $(x_{\alpha})_{\alpha\in J}\subseteq\Ualg_g$ be a net converging to an $x\in\Ualg$. Write $x=\sum_{h\in \G}x_h$. Let $(i,j)\in I$ be such that $e_{ii}x_he_{jj}\ne0$, for some $h$. Since the multiplication is continuous, $(e_{ii}x_\alpha e_{jj})_{\alpha\in J}$ is a net converging to $e_{ii}xe_{jj}=\sum_{h\in \G}e_{ii}x_he_{jj}$. However, $(e_{ii}x_\alpha e_{jj})_{\alpha\in J}$ is eventually constant (since $e_{ii}\Ualg e_{jj}=\F e_{ij}$, whose topology is discrete). Thus, for some $\alpha\in J$, $e_{ii}x_\alpha e_{jj}=\sum_{h\in \G}e_{ii}x_he_{jj}$. Since this equation involves only homogeneous elements, we should have
$$
g=\deg_G e_{ii}x_\alpha e_{jj}=\deg_G(e_{ii}x_he_{jj})=h.
$$
Therefore, $x\in\Ualg_g$.
\end{proof}

\begin{Cor}\label{inducedfromgood}
If \G is finite, then every good grading on $\bar{\Aalg}$ is induced from a good grading on \Aalg.
\end{Cor}
\begin{proof}
Let $\Gamma\colon\bar{\Aalg}=\bigoplus_{g\in \G}\Ualg_g$ be a finite good \G-grading on $\bar{\Aalg}$. Then, $\Aalg_g=\Ualg_g\cap\Aalg$ defines a good \G-grading $\Gamma_0$ on \Aalg. Since each $\Ualg_g$ is closed, one has $\Ualg_g\supseteq\bar{\Aalg_g}$. However, by  \Cref{goodiscontinuous}, $\overline{\Aalg}=\bigoplus_{g\in \G}\bar{\Aalg}_g$. Thus, $\Gamma=\overline{\Gamma_0}$.
\end{proof}

We proved that every good grading on \Aalg is continuous (\Cref{goodiscontinuous}). Also, an extension of a continuous grading is continuous (\Cref{extensioncontinuous}). Hence, as a consequence of  \Cref{inducedfromgood}, we also prove:

\begin{Cor}\label{cor:good:grading:finite:support}
Every good grading with finite support on $\bar{\Aalg}$ is continuous.\qed
\end{Cor}

Specializing the previous corollaries to our area of interest, we can state:

\begin{Cor}\label{goodclosedinduced}
Let $\mathcal{U}$ be an algebra, $\UTBVlim\subseteq\mathcal{U}\subseteq\UTBV$. Then, every finite good grading on $\mathcal{U}$ is continuous and induced from a finite good grading on \UTBVlim.\qed
\end{Cor}

\section{Classification of group gradings on triangularizable algebras}

In this section, we present our main results on the classification and properties of group gradings on upper triangular matrix algebras. We begin by demonstrating the equivalence of good and elementary gradings, and extend our analysis to infinite-dimensional settings.

In \Cref{sec_elem_grad}, we establish a correspondence between elementary \G-gradings on \UTBVlim and maps from $\bar{\B}=\B\setminus\{1\}$ to \G, where $1$ is the least element of \B. We define $\gamma_\Gamma(v) = \deg_\Gamma e_{1v}$ for all $v \in \bar{\B}$. This construction leads to a well-defined elementary \G-grading on \UTBVlim. We also prove that every good grading on \UTBV and \UTBVfin is finite.

In \Cref{sec_Prim_idemp}, we investigate homogeneous idempotents in arbitrary group gradings on \UTBV and show that any group grading is isomorphic to a good grading. We discuss the isomorphism classes of elementary group gradings, providing key insights into their classification.

The final subsection summarizes our main findings and explores the implications of our results, including a discussion on fine group gradings.

\subsection{Elementary grading\label{sec_elem_grad}}

We shall prove that there is a correspondence between elementary \G-gradings on \UTBVlim and maps $\bar{\B}\to \G$, where $\bar{\B}$ is obtained from \B by removing its first element (denoted here by $1$). As mentioned before, if $\Gamma$ is an elementary \G-grading on \UTBVlim, we define $\gamma_\Gamma(v)=\deg_\Gamma e_{1v}$, $\forall v\in\bar{\B}$.

Conversely, let $\gamma\colon\bar{\B}\to \G$ be a map. For convenience, we define $\gamma(1)=1$. Given $u\le v$ in \B, let
$$
\deg_Ge_{uv}=\gamma(u)^{-1}\gamma(v).
$$
Then, we obtain a well-defined elementary \G-grading on \UTBVlim. Indeed, first, $\UTBVlim=\mathrm{Span}\{e_{uv}\mid u\le v\in\B\}$, thus we obtain a vector space \G-grading on \UTBVlim. Given $u\le v\le w$, the equation $e_{uw}=e_{uv}e_{vw}$ is compatible with the vector space grading since
$$
\deg_Ge_{uv}\deg_Ge_{vw}=\gamma(u)^{-1}\gamma(v)\gamma(v)^{-1}\gamma(w)=\gamma(u)^{-1}\gamma(w)=\deg_Ge_{uw}.
$$
Then, we obtain a \G-grading on \UTBVlim, denoted by $\Gamma(\gamma)$. Clearly, by construction, the associated map is $\gamma_{\Gamma(\gamma)}=\gamma$. Thus, we obtain a bijective correspondence between the elementary \G-gradings and maps $\bar{\B}\to \G$.

We close this subsection with the following elementary remarks. Firstly, a \G-grading on \UTBVlim is good if and only if it is elementary. Indeed, an elementary grading is good by definition. Conversely, given a good \G-grading on \UTBVlim, we may define $\gamma\colon\B\to \G$ via
$$
\gamma(v)=\deg_\Gamma e_{1v}.
$$
Since $e_{1v}=e_{1u}e_{uv}$, we see that $\deg_\Gamma e_{uv}=\gamma(u)^{-1}\gamma(v)$. Therefore, a good grading may be described in terms of an elementary one.

Given $\gamma\colon\B\to \G$ and $g\in \G$, we define $g\gamma$ to be the map $g\gamma\colon\B\to \G$ such that $(g\gamma)(u)=g(\gamma(u))$, for all $u\in\B$. It is clear that $\gamma$ and $g\gamma$ define the same elementary grading on $\UTBlim$. We shall prove that the converse holds:

\begin{Lemma}\label{isomorphicelementarysequence}
Let $\eta$ and $\eta'$ be maps defining elementary \G-gradings $\Gamma(\eta)$ and $\Gamma(\eta')$ on $\UTBlim$. If $\Gamma(\eta)=\Gamma(\eta')$, then there exists $g\in \G$ such that $\eta=g\eta'$.
\end{Lemma}
\begin{proof}
Let $g=\eta(1)(\eta'(1))^{-1}$, where $1\in\B$ is the first element. Let $u\in\B$. Since $\Gamma(\eta)=\Gamma(\eta')$, one has
$$
\underbrace{\deg_{\Gamma(\eta)}e_{1u}}_{\eta(1)^{-1}\eta(u)}=\underbrace{\deg_{\Gamma(\eta')}e_{1u}}_{\eta'(1)^{-1}\eta'(u)}.
$$
Hence, $\eta(u)=g\eta'(u)$ for each $u\in\B$.
\end{proof}

Clearly, if we impose $\eta(1)=1=\eta'(1)$, then $\Gamma(\eta)=\Gamma(\eta')$ if and only if $\eta=\eta'$.

Finally, a \G-grading $\Gamma$ on \UTBVlim is good if every $e_{vv}$ is homogeneous in the grading. Indeed, for $u\le v$, $\Span\,e_{uv}=e_{uu}(\UTBVlim) e_{vv}$, thus each $e_{uv}$ is homogeneous.

All of the above discussion holds for group gradings on \UTBV as well as on any of its subalgebras containing \UTBVlim. In other words, we say that a \G-grading $\Gamma$ on \Aalg, where $\UTBVlim\subseteq\Aalg\subseteq\UTBV$, is \highlight{elementary} if it is good and there exists a map $\gamma\colon\B\to \G$ satisfying $\deg_\Gamma e_{uv}=\gamma(u)^{-1}\gamma(v)$, $\forall u,v\in\B$. The same argument as before shows that every good grading is an elementary grading. Now, we prove that every good grading on \UTBV and on \UTBVfin is finite.

\begin{Lemma}\label{lemma:goodgradingfinite}
Let $\Gamma$ be a good \G-grading on \Aalg, where $\UTBVfin\subseteq\Aalg\subseteq\UTBV$. Then $\Gamma$ is finite.
\end{Lemma}
\begin{proof}
From the previous discussion, we may assume that $\Gamma$ is elementary. Hence, there exists a map $\gamma\colon\B\to \G$ such that $\deg_\Gamma e_{uv}=\gamma(u)^{-1}\gamma(v)$, for each $u$, $v\in\B$. Let
$$
x=\sum_{v\in\B}e_{1v}\in\UTBVfin\subseteq\Aalg,
$$
where $1$ is the first element of \B. Denote $x=x_{g_1}+\cdots+x_{g_m}$, sum of homogeneous elements. Then, for each $v\in\B$, we have $e_{1v}=e_{11}xe_{vv}=\sum_{i=1}^me_{11}x_{g_i}e_{vv}$. Hence, comparing homogeneous degrees, one has
$$
\gamma(1)^{-1}\gamma(v)=\deg_\Gamma e_{1v}\in\{g_1,\ldots,g_m\}.
$$
Therefore, $\gamma(v)\in\{\gamma(1)g_1,\ldots,\gamma(1)g_m\}$, for all $v\in\B$. Thus, the image of $\gamma$ is a finite set, and $\Gamma$ is finite.
\end{proof}

\subsection{Primitive Homogeneous Idempotents\label{sec_Prim_idemp}}

Idempotents play a central role in this work. In this subsection, we establish several key properties of these elements, including the existence of primitive homogeneous idempotents in any \G-graded subalgebra \Aalg with $\UTBVlim \subseteq \Aalg \subseteq \UTBV$.

We begin by obtaining detailed information on idempotents derived from the least element of a well-ordered basis $(\B,\leq)$. This covers the structure of the associated algebraic entities without assuming that the ambient algebra \UTBV is graded.

\begin{Lemma}\label{lem:structure:corner:prim:idemp}
Let $(\B,\leq)$ be a well-ordered basis of \V and let \Aalg be a subalgebra of \UTBV. Assume that $\Gamma:\Aalg=\bigoplus_{g\in \G}\Aalg_g$ is a grading by a group \G. If $v\in\B$ is the least element and $e_{vv}\in\Aalg$, then the following holds:
\begin{enumerate}
\item\label{lem:structure:corner:prim:idemp:A} There exists a primitive idempotent $f_v\in\Aalg_1$ with $v\in\Supp_{\B}(f_v)$;
\item\label{lem:structure:corner:prim:idemp:B} The corner subalgebra $\Dalg=f_v\Aalg f_v$ is a graded commutative division algebra with identity $f_v$;
\item\label{lem:structure:corner:prim:idemp:C} There exists a finite subgroup $H$ of \G such that $\Dalg\cong\F H$ as \F-algebras;
\item\label{lem:structure:corner:prim:idemp:D} $|H|$ is a unit in \F and \Dalg is semisimple;
\item\label{lem:structure:corner:prim:idemp:E} The left annihilator $\ann_{\Aalg}\Dalg$ is a two-sided graded ideal equal to $\Span\{a\in\Aalg\text{ homogeneous}\mid a(v,v)=0\}$.
\end{enumerate}
\end{Lemma}
\begin{proof}
We proceed in steps:

\textbf{Step 1: Decompose $e_{vv}$}. Let $e_{vv}=\gamma_1+\cdots+\gamma_n$, with $\gamma_i\in\Aalg_{h_i}$,
be the homogeneous decomposition of $e_{vv}$, with $h_1,\ldots,h_n\in \G$ distinct. Since $e_{vv}\neq 0$, we may assume that, for some $p\geq 1$, $\gamma_i(v,v)\neq 0$ for $i=1,\ldots,p$ and $\gamma_i(v,v)=0$ for $i=p+1,\ldots, n$ (if any). 

\textbf{Step 2: $H=\{h_1,\ldots,h_p\}$ is a subgroup of \G and $\Dalg=\Span\{\gamma_1,\ldots,\gamma_p\}$ is a graded subalgebra of \Aalg}.
For $t\in\Aalg$ arbitrary, by the minimality of $v$, we have 
$t e_{vv}=t(v,v)e_{vv}$,
hence
\[t\gamma_1+\cdots+t\gamma_n=t(v,v)\gamma_1+\cdots + t(v,v)\gamma_n.\]
For a homogeneous $t\in\Aalg_g$, this equation gives either $t\gamma_j=0$ for all $j$ or $0\neq t\gamma_j=t(v,v)\gamma_{\sigma(j)}\neq 0$ and $gh_j=h_{\sigma(j)}$ for all $j$ and some permutation $\sigma\in S_n$. In turn, this identity gives $(t\gamma_j)(v,v)=t(v,v)\gamma_j(v,v)=t(v,v)\gamma_{\sigma(j)}(v,v)$.
Hence, if $t(v,v)\neq 0$, then we have $1\leq j\leq p$ if, and only if, $1\leq \sigma(j)\leq p$.

In particular, for $t=\gamma_i$ ($1\leq i\leq p$), we obtain 
\begin{equation}\label{eq:D:struc:rels}
0\neq \gamma_i\gamma_j=\gamma_i(v,v)\gamma_{\sigma_i(j)}
\end{equation}
for some permutation $\sigma_i\in S_n$ for all $1\leq j\leq n$, and thus $h_ih_j=h_{\sigma_i(j)}\in H$ for $1\leq j\leq p$. It follows that $H$ is a finite subgroup of \G. Relabeling if necessary, we may assume that $h_1=1$. Also, \cref{eq:D:struc:rels} implies that $\Dalg$ is a graded subalgebra of \Aalg with homogeneous generators $\gamma_1,\ldots,\gamma_p$.

\textbf{Step 3: The element $f_v=\gamma_1(v,v)^{-1}\gamma_1$ is a homogeneous idempotent and the identity of $\Dalg$}. Since $\gamma_1(v,v)\neq 0$, it follows from \cref{eq:D:struc:rels} that the element $\gamma_1$ satisfies $\gamma_1 \gamma_j=\gamma_1(v,v)\gamma_j$ for all $1\leq j\leq n$ and $\gamma_j\gamma_1=\gamma_j(v,v)\gamma_j$ for all $1\leq j\leq p$. In particular, $\gamma_1^2=\gamma_1(v,v)\gamma_1$, hence $f_v=\gamma_1(v,v)^{-1}\gamma_1$ is a homogeneous idempotent, acting as the identity on the elements $\gamma_1,\ldots,\gamma_p$. Therefore, $f_v$ is the identity of $\Dalg$.

\textbf{Step 4: $\Dalg=f_v\Aalg f_v$ is a finite dimensional graded division algebra.} Consider the corner subalgebra $f_v\Aalg f_v$, which is obviously graded. On the one hand, from Step 3, for all $j\leq p$, we have  $\gamma_j=\gamma_j(v,v)^{-1}\gamma_j\gamma_1=\gamma_j(v,v)^{-1}\gamma_1(v,v)f_v \gamma_j f_v\in f_v\Aalg_{h_j} f_v$, hence $\Dalg \subseteq f_v\Aalg f_v$.

On the other hand, let $t\in f_v\Aalg_g f_v $ be a homogeneous element for some $g\in \G$. From Step 2, if $t(v,v)=0$, then $t=tf_v=0$, else
$t=tf_v=t(v,v)\gamma_1(v,v)^{-1}\gamma_{\sigma(1)}$ for some permutation $\sigma\in S_n$ with $\sigma(1)\leq p$. It follows that $f_v\Aalg f_v\subseteq\Dalg$. 

Therefore, $\Dalg = f_v\Aalg f_v$ is a finite-dimensional graded division algebra. 

\textbf{Step 5: $f_v$ is primitive in $\Aalg_1$.} Suppose that $e\in\Aalg_1$ is an idempotent with $e\leq f_v$. Then $e = e f_v = \gamma_1(v,v)^{-1}\gamma_1 = \gamma_1(v,v)^{-1}e(v,v)\gamma_1 = e(v,v)f_v$. Since $e(v,v)\in\{0,1\}$, it follows that $e=0$ or $e=f_v$, hence $f_v$ is primitive in $\Aalg_1$.

\textbf{Step 6: $\Dalg\cong\F H$ as graded algebra and it is semisimple}. 
First notice that from \cref{eq:D:struc:rels} we have $\gamma_i(v,v)\gamma_j(v,v)=\gamma_i(v,v)\gamma_{\sigma_i(j)}(v,v)$, hence $\gamma_j(v,v)=\gamma_{\sigma_i(j)}(v,v)$ for all $i,j\leq p$. In particular, for $j=1$, we have $\gamma_i(v,v)\gamma_1(v,v)=\gamma_i(v,v)\gamma_i(v,v)$, hence $\gamma_i(v,v)=\gamma_1(v,v)$ for all $i\leq p$. It follows that $1=e_{vv}(v,v)=p\gamma_1(v,v)+(n-p)0$, hence $p=|H|=\dim\Dalg=\gamma_1(v,v)^{-1}$ is invertible in \F.

Next, let $\varphi:\Dalg\to\F H$ be the linear isomorphism defined by $\gamma_i\mapsto \gamma_i(v,v)h_i$ for $i\leq p$. It satisfies \begin{align*}\varphi(\gamma_i\gamma_j)&=\gamma_i(v,v)\varphi(\gamma_{\sigma_i(j)})=\gamma_i(v,v)\gamma_{\sigma_i(j)}(v,v)h_{\sigma_i(j)}\\
&=\gamma_i(v,v)h_i\gamma_j(v,v)h_j=\varphi(\gamma_i)\varphi(\gamma_j),
\end{align*}
and $\varphi(f_v)=\gamma_1(v,v)^{-1}\varphi(\gamma_1)=h_1=1$,
hence $\varphi$ is an algebra isomorphism between $\Dalg$ and $\F H$. Finally, since $|H|$ is invertible in \F, it follows from Maschke's Theorem that $\F H$ is semisimple.

\textbf{Step 7: $\ann_{\Aalg}\Dalg$ equals $\Span\{a\in\Aalg\text{ homogeneous}\mid a(v,v)=0\}$}. Let $a\in\Aalg$ be homogeneous. If $a(v,v)$=0, then $a e_{vv}=0$, hence $a\gamma_i=0$ for all $i\leq p$, so $a\in\ann_{\Aalg}\Dalg$.
Conversely, if $a\in\ann_{\Aalg}\Dalg$ is homogeneous, then $a\gamma_1=0$, hence $a(v,v)\gamma_1(v,v)=0$, so $a(v,v)=0$.

The proof is now complete.
\end{proof}

The next example shows that, even though $\Dalg=f_v\Aalg f_v$ is finite-dimensional, $\Supp_{\B}(f_v)$ does not need to be finite.
\begin{Example}\label{ex:infinite:support}
Assuming that $2\neq 0$ in \F, let $\B=\{v_i\mid i\in\NN\}$ and let $\gamma_1,\gamma_2$ be elements of \UTBV defined as $\gamma_1=1/2$ and
\[\gamma_2(i,j)=\begin{cases}1/2,&i=j=1\\-1/2,&i=j\neq 1\\0,&\text{otherwise.}\end{cases}\]
Let \Aalg be the subalgebra of \UTBV generated by $\gamma_1,\gamma_2$, graded by a group \G, with $\deg\gamma_1=1$ and $\deg\gamma_2=g$, where $g^2=1$. Then $e_{11}=\gamma_1+\gamma_2\in\Aalg$, and we may apply the previous Lemma. This gives $f_{1}=2\gamma_1$ a primitive idempotent and $\Dalg=\Aalg\cong \F\ZZ_2$ semisimple graded but, as we can see, $\Supp_{\B}f_{1}=\B$, which is not finite.
\end{Example}

If the grading on \Aalg is consistent with the grading on \UTBV, additional conclusions can be drawn:

\begin{Prop}\label{prop:existence:primitive:idemp}
Let $(\B,\leq)$ be a well-ordered basis of \V and assume that \UTBV is graded by a group \G. If $v$ is the least 
element in \B and $\Aalg\subseteq\UTBV$ is a graded subalgebra containing $e_{vv}$, then there exists a primitive idempotent $f_v$ in $\Aalg_1$ with $\Supp(f_v)=\{v\}$.
\end{Prop}
\begin{proof}
From \Cref{lem:structure:corner:prim:idemp}, there exists a primitive idempotent $f_v\in\Aalg_1$ with $v\in\B'=\Supp_{\B}(f_v)$ which remains primitive in $(\UTBV)_1$.

From \Cref{prop:triang:corner}, \(f_v(\UTBV) f_v \cong \UT_{\B'} \V'\), where \(\V' = \Span \B'\), and this algebra inherits the grading. Again, from \Cref{lem:structure:corner:prim:idemp} applied to \(\UT_{\B'} \V'\), \(v\), and \(e_{vv}\), there exists a primitive idempotent \(e_v \in (\UT_{\B'} \V')_1\) and a commutative graded division subalgebra \(\Dalg' = e_v(\UT_{\B'} \V')e_v\). Since 1 (the image of \(f_v\) under the isomorphism) is the only nonzero homogeneous idempotent, we get \(e_v = 1\) and \(\Dalg' = \UT_{\B'} \V'\). Given that \(\Dalg'\) is semisimple, it follows that \(|\B'| = 1\). Therefore, \(\Supp_{\B}(f_v) = \{v\}\), as required.
\end{proof}

\begin{Lemma}\label{cor:supp:primitive:idemp}
Let $(\B,\leq)$ be a well-ordered basis of \V and assume that \UTBV is graded by a group \G. Let $\UTBVlim\subseteq\Aalg\subseteq\UTBV$ be a graded subalgebra, and let $f$ be a homogeneous idempotent in $\Aalg_1$. Then $f$ is primitive in $\Aalg_1$ if, and only if, $\Supp_{\B}(f)$ is a singleton.
\end{Lemma}
\begin{proof}
Let $f$ be a primitive idempotent in $\Aalg_1$. From \Cref{prop:triang:corner}, we have that $f\Aalg f$ is
isomorphic as an \F-algebra to a subalgebra $\Aalg'$ of $\UT_{\B'}\V'$, where $\B'=\Supp_{\B}(f)$ and $\V'=\Span\B'$. Given that $f$ is homogeneous,   
$\Aalg'$ inherits the grading. From \Cref{prop:existence:primitive:idemp}, there is a primitive idempotent $f_v\in\Aalg_1'$ with $\Supp_{\B'}(f_v)=\{v\}$. The preimage $g$
of $f_v$ under the isomorphism is a nonzero idempotent in $\Aalg_1$ satisfying $gf=g=fg$. Therefore, due to primitivity, $g$ must equal $f$. It is clear that the support of $f$ is a singleton. The converse is obvious.
\end{proof}

\begin{Def}\label{def:complete:idempotents}
A set \Eset of nonzero orthogonal idempotents in  \UTBV is \highlight{complete} if $\bigcup_{e\in\Eset}\mathrm{Supp}(e)=\B$.
\end{Def}

If \Eset is complete, then $\sum_{e\in\Eset}e(\V)=\V$, making the set summable (cf. \cite[Lemma 2.11(b)]{IMesR}) and summing to the identity transformation in \UTBV. Conversely, if \Eset is summable and sums to the identity in \UTBV, it is necessarily complete.

We will show that any graded subalgebra $\UTBVlim\subseteq\Aalg\subseteq\UTBV$ contains a complete set of primitive homogeneous orthogonal idempotents. First, we will consider the special case where \B is indexed by \NN.

\begin{Lemma}\label{lem2}
Let $\B=\{v_i\mid i\in\NN\}$. Then, there exists a complete set of primitive pairwise orthogonal homogeneous idempotents $\Eset\subseteq\Aalg_1$, where $\UTBVlim\subseteq\Aalg\subseteq\UTBV$ is a graded subalgebra.
\end{Lemma}
\begin{proof}
If $1\notin\Aalg$, we may replace \Aalg with $\Aalg^\sharp=\F1\oplus\Aalg$, and assume that $1$ is homogeneous of trivial degree. Note that $\Aalg\subseteq\Aalg^\sharp$ is a graded subalgebra. First, we use \Cref{prop:existence:primitive:idemp} (see also \Cref{cor:supp:primitive:idemp}) to find a primitive homogeneous idempotent $e_1\in\Aalg_1$ and define $\Eset_1=\{e_1\}$. Next, assume that the set of primitive pairwise orthogonal homogeneous idempotents $\Eset_m$ with $m$ elements has been constructed for some integer $m>1$, and consider the element $f=1-\sum_{i=1}^m e_i$, which is a homogeneous idempotent in $\Aalg_1$. According to \Cref{cor:corner:A}, the corner subalgebra $f\Aalg f$ is isomorphic to a subalgebra $\Aalg'$ with $\UTmlim\V_{m}\subseteq\Aalg'\subseteq\UTm\V_{m}$, where $\B_m=\{v_i\mid i\geq m+1\}$ and $\V_{m}=\Span\B_m$, and this subalgebra is still graded. A further application of \Cref{prop:existence:primitive:idemp} to $\Aalg'$ yields a primitive homogeneous idempotent $e_{m+1}\in f\Aalg_1 f$, which is orthogonal to every element in $\Eset_m$ by construction. We then define $\Eset_{m+1}=\Eset_m\cup\{e_{m+1}\}$. To finish the proof, we define $\Eset=\bigcup_{m\geq 1}\Eset_m$. Since every $v\in\B$ belongs to the support of exactly one element in \Eset (see \Cref{cor:supp:primitive:idemp}), we have $\bigcup_{i\geq 1} \Supp_{\B}(e_i)=\B$, hence \Eset is complete.
\end{proof}

The next lemma addresses the existence of homogeneous idempotents containing a given subset in their their support. Specifically, it establishes the existence of a primitive homogeneous idempotent whose support is a given singleton.

\begin{Lemma}\label{prop:existence:e_v:prim}
Let $(\B,\leq)$ be a well-ordered basis for \V and assume that \UTBV is \G-graded. If $\emptyset\neq S\subseteq \B$ is arbitrary, then $(\UTBV)_1$ contains an idempotent $f$ with $S\subseteq\Supp_{\B}(f)$ and these sets share the same least element. In particular, given any $v\in\B$, there exists a primitive idempotent $f\in(\UTBV)_1$ with $\Supp_{\B}(f)=\{v\}$.
\end{Lemma}

\begin{proof}
Consider the set,
\[\Eset_{S}=\{e\in(\UTBV)_1\mid e^2=e,\,S\subseteq\Supp_{\B}(e)\}\]
ordered by the usual partial order on idempotents, namely $e\leq f$ $\iff$ $ef=e=fe$. Notice that $1\in\Eset_{S}$, hence this set is nonempty. If $\mathcal{C}\subseteq\Eset_{S}$ is a totally ordered set, we may regard  $\mathcal{C}$ as a net. If $e\leq f$ in $\mathcal{C}$, it is clear that $e(\V)\subseteq f(\V)$ and $\ker e\supseteq \ker f$. Define $\W_0=\bigcup_{e\in\mathcal{C}}\ker e$ and $\W_1=\bigcap_{e\in\mathcal{C}}e(\V)$, so that $\V=\W_0\oplus\W_1$. 
Now let $g$ be the element of \EndFV acting as the identity on $\W_1$ and with kernel $\W_0$, namely the projection onto $\W_1$, making $g$ an idempotent. For any $w\in\W_1$ and all $e\in\mathcal{C}$, we have $e(w)=w=g(w)$. For $w\in\W_0$ there exists some $f_0\in\mathcal{C}$ with $e(w)=0=g(w)$ for all $e < f_0$. In any case, given a finite-dimensional vector subspace $\mathcal{U}\subseteq\V$ satisfying $\mathcal{U}=\W_0\cap \mathcal{U}\oplus\W_1\cap\mathcal{U}$, there is $f_0\in\mathcal{C}$ such that $e|_\mathcal{U}=g|_\mathcal{U}$ for all $e<f_0$. It follows that the net $\mathcal{C}$ converges to the element $g\in\UTBV$. It is clear that $S\subseteq\Supp_{\B}(g)$.
Since multiplication is a continuous operation,
for all $e\in\mathcal{C}$, we have $eg=g=ge$.

We claim that $g$ is a lower bound for $\mathcal{C}$ in $\Eset_{S}$, 
hence this set contains a minimal element, say $f$. To prove this claim, let $g=\epsilon_1+\cdots+\epsilon_n$ be a homogeneous decomposition with $\epsilon_i\in(\UTBV)_{h_i}$ for $h_i\in \G$ distinct. Here we allow some of the $\epsilon_i$ to be zero. From the previous paragraph, for any $e\in\mathcal{C}$, we have $eg=g$, hence $e\epsilon_1+\cdots+e\epsilon_n=\epsilon_1+\cdots+\epsilon_n$. Comparing terms of equal degree, we have $e\epsilon_i=\epsilon_i$ for all $i$. Similarly, we also have $\epsilon_i e=\epsilon_i$ for all $i$. Passing these identities to the limit, we get $g\epsilon_i=\epsilon_i=\epsilon_ig$ for all $i$. If $\epsilon_i\neq 0$, then there exists $j$ such that $\epsilon_j\epsilon_i=\epsilon_i$ and $\epsilon_k\epsilon_i=0$ for all $k\neq j$. This further implies that $\epsilon_j\neq 0$ and $h_jh_i=h_i$, hence  that $h_j=1$. Without loss of generality, we may assume that $j=1$ and $h_1=1$. Now from $g\epsilon_1=\epsilon_1$, we get $\epsilon_1^2=\epsilon_1$ and $\epsilon_i\epsilon_1=0$ for $i>1$. Similarly, from $\epsilon_ig=\epsilon_i$, we get $\epsilon_i\epsilon_1=\epsilon_i$, for all $i$, hence $\epsilon_2=\cdots=\epsilon_n=0$. It follows that $g=\epsilon_1\in(\UTBV)_1$ and that $g$ is a lower bound for $\mathcal{C}$ in $\Eset_{S}$, as claimed.

Next, let $u\in\Supp_{\B}(f)$ be the least element. Define $\B'=\Supp_{\B}(f)$ and $\V'=\Span\B'$, so that $f(\UTBV)f\cong\UT_{\B'}\V'$. Then \Cref{prop:existence:primitive:idemp} implies that there exists a primitive idempotent $f_u\in f(\UTBV)_1 f$ with $\Supp_{\B'}(f_u)=\{u\}$. If $u$ is not the least element in $S$, then $p=f-f_u$ becomes an idempotent in $(\UTBV)_1$, with $p < f$ and $S\subseteq\Supp_{\B}(p)$, hence $p\in\Eset_{S}$. But this contadicts the minimality of $f$, therefore $S$ and $\Supp_{\B}(f)$ must share the same least element, $u$.

If $S=\{v\}$ then $u=v$, and $f_v\in\Eset_{S}$ with $f_v\leq f$, hence $f=f_v$, and this completes the proof. 
\end{proof}

\begin{Thm}\label{thm:existence:complete:homogeneous:idemps}
Let $(\B,\leq)$ be a well-ordered basis of \V, and assume that \UTBV is \G-graded. Then $(\UTBV)_1$ contains a complete set of primitive pairwise orthogonal idempotents whose supports are singletons.
\end{Thm}

\begin{proof}
If \B contains no limit ordinals, then \B is in bijective correspondence with a subset of \NN. From \Cref{lem2}, there exists in $(\UTBV)_1$ a complete set \Eset of primitive pairwise orthogonal idempotents whose supports are singletons.
 
We may then assume that $\omega\in\B$ is the first limit ordinal and define $S=\{u\in\B\mid u\geq \omega\}$. From \Cref{prop:existence:e_v:prim}, there exists an idempotent $f\in(\UTBV)_1$ with $\Supp_{\B}(f)=S$.

Since $\B'=\Supp_{\B}(1-f)$ contains no limit ordinals, and $(1-f)\UTBV(1-f)\cong\UT_{\B'}\V'$, where $\V'=\Span(\B')$ is still graded, again from \Cref{lem2}, the subalgebra $(1-f)(\UTBV)_1(1-f)$ contains a complete set $\Eset_1$ of primitive pairwise orthogonal idempotents whose supports are singletons. We check that $\Eset_1$ is summable in \UTBV. Indeed, let $\B''=\{f(v)\mid v\in\Supp_{\B}(f)\}$, so that $f(w)=w\neq 0$ for all $w\in\B''$. Let $\V''=\Span(\B'')$. For all $w\in\V''$ and all $e\in\Eset_1$ we have $e(w)=ef(w)=0$, so
\[\bigoplus_{e\in\Eset_1}\range(e)\bigoplus (\bigcap_{e\in\Eset_1}\ker e)\supseteq \V'\oplus \V''=\V,\]
hence $\Eset_1$ is summable (cf. \cite[Lemma 2.11(b)]{IMesR}). 
We can thus take the idempotent $g=1-\sum_{e\in\Eset_1}e$, whose support is still $\B''$. As in the proof of \Cref{prop:existence:e_v:prim}, $g\in(\UTBV)_1$, and $g$ is orthogonal to each idempotent in $\Eset_1$. 
Therefore, by considering the subalgebra $g(\UTBV)_1g\cong\UT_{\B''}\V''$, it follows by transfinite induction on the ordinal $\B''$ that this algebra contains a complete set $\Eset_2$ of primitive pairwise orthogonal idempotents whose supports are singletons.

It is clear that the elements in $\Eset_2$ are orthogonal to every element in $\Eset_1$. Also, $\bigcup_{e\in\Eset_1\cup\Eset_2}\Supp_{\B}(e)=\B'\cup\B''=\B$, so by taking $\Eset=\Eset_1\cup\Eset_2$, the \namecref{thm:existence:complete:homogeneous:idemps} follows.
\end{proof}

Finally, we shall prove that a complete set of primitive pairwise orthogonal idempotents is simultaneously diagonalizable and derive a few implications of these findings. 

\begin{Lemma}\label{lem3}
Let $(\B,\leq)$ be a well-ordered basis for \V, and let \Eset be a complete set of primitive pairwise orthogonal homogeneous idempotents of \Aalg, where $\UTBVlim\subseteq\Aalg\subseteq\UTBV$ is a subalgebra. Then there exists a well-ordered basis $\B'$ of \V with respect to which the elements in \Eset are diagonal and $\UTBV=\UT_{\B'}V$.
\end{Lemma}
\begin{proof}
From \Cref{cor:supp:primitive:idemp}, it follows that the supports of the elements in \Eset are singletons, hence this set is in bijective correspondence with \B. Furthermore, we may assume that \Eset is indexed by the vectors in \B, so that $\Supp_{\B}(e_v)=\{v\}$, for all $v\in\B$.

If we define $w_v=e_v(v)$, then $w_v\neq 0$ and the set $\B'=\{w_v\mid v\in \B\}$ is linearly independent due to orthogonality. Moreover, since $w_v$ spans $e_v(\V)$ and $\sum_{v\in\B}e(\V)=\V$, the set $\B'$ generates \V.
It follows that $\B'$ is a basis of \V with respect to which the elements of \Eset are diagonal. It is clear that $\B'$ inherits the ordering of \B, hence it is well-ordered as well.

Finally, we shall show that \UTBV is triangular with respect to $\B'$. Specifically, we shall prove that $\W_v=\Span\{w_u\mid u\le v\}=\Span\{u\mid u\le v\}=:\V_v$ for each $v\in\B$. It is evident that the equality holds for the least element of \B. Now, let $v\in\B$ and assume that $\W_u=\V_u$ for each $u<v$. Then one obtains $\Span\{w_u\mid u<v\}=\mathrm{Span}\{u\mid u<v\}$. Hence, the codimension of $\Span\{w_u\mid u<v\}$ in $\V_v$ is $1$. By definition of \UTBV, $e_v(v)\in\V_v$, so $\W_v\subseteq\V_v$. This implies the equality $\W_v=\V_v$. Transfinite induction guarantees the claim.

Therefore, given $x\in\UTBV$ and $v\in\B$, one has $w_v\in\W_v=\V_v$ and
$$
x(w_v)\in\V_v=\W_v.
$$
Hence, $x$ is triangular with respect to $\B'$.
\end{proof}

The following theorem extends a result originally established in \cite{VZ2007} for finite upper triangular matrices.

\begin{Thm}\label{thm:grading:class:utbv}
Let $(\B,\leq)$ be a well ordered basis of \V. Then, every group grading on \UTBV is isomorphic to an elementary grading.
\end{Thm}
\begin{proof}
By \Cref{thm:existence:complete:homogeneous:idemps}, there exists a complete set \Eset of primitive pairwise orthogonal idempotents in $(\UTBV)_1$. Furthermore, by \Cref{lem3}, there is a well-ordered basis $\B'$ with respect to which the elements in \Eset are $e_{vv}$ for $v\in\B'$. It follows from the comments at the end of \Cref{sec_elem_grad} that the grading on $\UT_{\B'}\V=\UTBV$ induced by the automorphism sending \B to $\B'$ is good and hence elementary.
\end{proof}

\begin{Cor}\label{cor:closed:ideals}
If \UTBV is graded and \Ialg is a closed ideal in the function topology, then \Ialg is a graded ideal.
\end{Cor}
\begin{proof}
In view of the theorem, we may assume that the grading is elementary, so the elements $e_{uv}$ are homogeneous in the grading for all $u,v\in\B$. Let $r\in\Ialg$ and let $r=\varrho_1+\cdots+\varrho_n$ be the homogeneous decomposition. If $e_{uu}re_{vv}=r(u,v)e_{uv}=0$, for some $u,v\in\B$, then $0=e_{uu}\varrho_1 e_{vv}+\cdots+e_{uu}\varrho_ne_{vv}$ is a homogeneous decomposition of zero. Therefore, $\varrho_i(u,v)e_{uv}=e_{uu}\varrho_i e_{vv}=0$, implying $\varrho_i(u,v)=0$ for all $i$. This means that if $\varrho_i(u,v)\neq 0$ for some $i$ then $r(u,v)\neq 0$, and this can occur for at most one $i$, namely the one for which $\deg\varrho_i=\deg e_{uv}$. Thus, $0\neq e_{uu}re_{vv}=r(u,v)e_{uv}=\varrho_i(u,v)e_{uv}=e_{uu}\varrho_ie_{vv}\in \Ialg$, with $\deg\varrho_i=\deg e_{uv}$, while $0=e_{uu}\varrho_je_{vv}\in \Ialg$ for all $j\neq i$. In any case, we have $e_{uu}\varrho_ie_{vv}\in\Ialg$ for all $u,v\in\B$.

Since the set $\{e_{vv}\mid v\in\B\}$ is summable to $1$ and multiplication is a continuous operaton on \UTBV, we have
\[\varrho_i=1\varrho_i1=\sum_{u\leq v\in\B}e_{uu}\varrho_ie_{vv}\]
for all $i$. But $\Ialg$ is closed and each term in this sum is in $\Ialg$, hence $\varrho_i\in\Ialg$ for all $i$. It follows that $\Ialg$ is a graded ideal.
\end{proof}

The Jacobson Radical $J(\UTBV)$ is a closed ideal of \UTBV (cf. \cite[Theorem 3.8(3)]{IMesR}); therefore, it is graded. Here, we provide a direct proof of this fact that does not rely on the closedness of this ideal. Instead, we rely on the fact that, since \UTBV is topologically closed, $J(\UTBV)$ equals $\TNil(\UTBV)$, the set of topologically nilpotent elements, which, in turn, equals the set of strictly triangular elements of \UTBV.

\begin{Cor}\label{cor:JR:graded}
If \UTBV is graded then the Jacobson Radical $J(\UTBV)$ is a graded ideal equal to $\TNil(\UTBV)$.
\end{Cor}

\begin{proof}
We may assume that the group grading is elementary. The second part of the statement follows from \cite[Proposition 5.4]{Mesb}, since \UTBV is closed in the function topology of \EndFV, and hence $J(\UTBV)$ is precisely the set of all strictly upper triangular transformations with respect to \B.

Let $r\in J(\UTBV)$ and let $r=\varrho_1+\cdots+\varrho_n$ be a homogeneous decomposition. Then
\[0=r(v,v)e_{vv}=e_{vv}re_{vv}=e_{vv}\varrho_1e_{vv}+\cdots+e_{vv}\varrho_ne_{vv}\]
is a homogeneous decomposition of zero. Hence, $\varrho_i(v,v)e_{vv}=e_{vv}\varrho_i e_{vv}=0$ for all $i$. It follows that $\varrho_i(v,v)=0$ for all $i$ and all $v\in\B$, so $\varrho_i$ is strictly upper triangular, hence $\varrho_i\in J(\UTBV)$ for all $i$. Therefore, $J(\UTBV)$ is a graded ideal, as required.
\end{proof}

As an immediate consequence we have:

\begin{Cor}\label{cor:utbv:quotient}
If \UTBV is graded, then the quotient $\UTBV/J(\UTBV)\cong \F^\B$ bears the trivial grading.
\end{Cor}
\begin{proof}
The map $\varphi:\UTBV\to\F^\B$ defined as $t\mapsto (t(v,v))_{v\in\B}$ is a continuous, open and surjective algebra map with kernel $\TNil(\UTBV)=J(\UTBV)$ (cf. \cite[Proposition 6.3]{Mesb}). Therefore, $\UTBV/J(\UTBV)\cong\F^\B$ as topological \F-algebras and, as a quotient of graded objects, this isomorphism induces the trivial grading on $\F^\B$.
\end{proof}
For any closed graded subalgebra \Aalg of \UTBV, this result remains true with a similar proof, if \B is replaced by an adequate ordinal. However, as shown in the following example, this \namecref{cor:utbv:quotient} can be false if the grading on \Aalg is not related to that of \UTBV.
\begin{Example} If \Aalg is the subalgebra described in \Cref{ex:infinite:support}, then $J(\Aalg)=0$ is clearly a graded ideal, but $\Aalg/J(\Aalg)\cong\F^2$ has a nontrivial $\ZZ_2$-grading. Of course, this grading on \Aalg cannot be extended to a grading on \UTBV.
\end{Example}

Next, we address the case of subalgebras of elements in \UTBV with finite rank. We will need the following lemma.

\begin{Lemma}\label{lem1}
Let \Aalg be a graded subalgebra of \UTBVfin, and let $f\in\EndFV$ be an idempotent such that $f\Aalg$ is a graded subspace. For each $v\in \B$ such that $e_{vv}Te_{vv}\neq 0$ for some $T\in f\Aalg$, there exists a homogeneous idempotent $e\in f\Aalg$ with $\Supp_{\B}(e)=\{v\}$.
\end{Lemma}

\begin{proof}
Let $t\in f\Aalg$ be homogeneous and such that $e_{vv}te_{vv}\ne0$ (i.e., $t$ is one of the components in the homogeneous decomposition of $T$ with respect to the group grading on $f\Aalg$).

By \Cref{JCD}, $t$ is algebraic over \F, and the powers $t^n$ ($n\in\NN$) are nonzero, homogeneous, and span a finite-dimensional subspace of $f\Aalg$, hence $\deg_G t$ is a torsion element of \G. Then, replacing $t$ with one of its powers, if necessary, we may assume that $\deg_G t=1$.

It follows from \Cref{JCD} that there exists an idempotent element $\epsilon$ in the (finite-dimensional) subspace of \EndFV spanned by $\{t,t^2,\ldots\}$ such that $e_{vv} \epsilon e_{vv}\neq 0$. Obviously $\epsilon$ is also homogeneous of degree 1 and lies in $f\Aalg$.

Finally, $\epsilon\Aalg \epsilon$ is a \G-graded algebra isomorphic to (a graded subalgebra of) $\UT_s\F$ by \Cref{cor:corner:A}, where $s=|\Supp(\epsilon)|$. Since every \G-grading on $\UT_s\F$ is elementary (see the main result of \cite{VZ2007} or our main result, \Cref{thm:existence:complete:homogeneous:idemps}), we find a (complete) set of $s$ orthogonal homogeneous idempotents in $\UT_s\F$ whose supports are singletons among which we are sure to find an idempotent $e$ whose support is $\{v\}$. By construction, such an idempotent is homogeneous and belongs to $f\Aalg$. The proof is complete.
\end{proof}

\begin{Lemma}\label{lem:idemp:keep:grad}
Let $(\B,\leq)$ be a well-ordered basis of \V and assume that \UTBVfin is graded by a group \G. If \Aalg is a graded subalgebra of \UTBVfin containing \UTBVlim, then given any nonempty subset $S\subseteq\B$, there exists an idempotent $f\in\UTBV$ such that $f\Aalg f$ is still graded and $S=\Supp_{\B}(f)$.
\end{Lemma}
\begin{proof}
The proof is similar to that of \Cref{prop:existence:e_v:prim} if we consider the subset $$\Eset_S=\{f\in\UTBV\mid f^2=f,\,S\subseteq\Supp_{\B}(f),\,f\Aalg f\subseteq\Aalg\text{ graded}\},$$
partially ordered as usual. Of course $1\in\Eset_S$, and this set is nonempty.
As in the proof of that result, given a totally ordered subset $\mathcal{C}$ of $\Eset_S$, there is a limiting idempotent $f\in\UTBV$ (the limit of the net $\mathcal{C}$). This element satisfies $fe=f=ef$ for all $e\in\mathcal{C}$ and $S\subseteq\Supp_{\B}(f)$. 

To show that $\Eset_S$ has minimal elements, it is enough to prove that $f\in\Eset_S$, which shall be the case if $f\Aalg f=\bigoplus_{g\in \G} f\Aalg_g f$ is a group grading. To this end, we observe that if $a_g\in f\Aalg_g f$ is homogeneous, then $ea_g e=e(f a_g f)e = f a_g f = a_g$, hence $f\Aalg_g f\subseteq e\Aalg_g e$ for all $e\in\mathcal{C}$. In particular, $f\Aalg f\subseteq\Aalg$.
By hypothesis, we have $(e\Aalg_g e)(e\Aalg_h e)\subseteq e\Aalg_{gh}e$, hence $(f\Aalg_g f)(f \Aalg_h f)\subseteq e\Aalg_{gh}e$, for all $e\in\mathcal{C}$.
It follows that $(f\Aalg_g f)(f\Aalg_h f)\subseteq f\Aalg_{gh} f$, for all $g,h\in \G$. Therefore, $\Eset_S$ contains a minimal element, which we shall continue to denote by $f$.

Finally, suppose that $S\subsetneq\Supp_{\B}(f)$, and let $u\in\Supp_{\B}(f)\setminus S$. It follows from \Cref{lem1} (see also \Cref{cor:corner:A}) that $f\Aalg f$ contains a primitive idempotent $e_u\in f\Aalg_1 f$ with $\Supp_{\B}e_u=\{u\}$. But then the element $p=f-e_u$ is an idempotent with $S\subseteq\Supp_{\B}p$ and $p\Aalg p\subset\Aalg$ is still graded,
but this contradicts the minimality of $f$. Therefore, $S=\Supp_{\B}(f)$, as required.  
\end{proof}

The following result classifies group gradings on subalgebras of \UTBVfin containing \UTBVlim without assuming a group grading on \UTBV.

\begin{Thm}\label{thm:grading:class:utbvfin}
Let $(\B,\leq)$ be a well-ordered basis and assume that \UTBVfin is graded by an arbitrary group \G. Then the grading is isomorphic to a good grading.
\end{Thm}

\begin{proof}
Let $\Aalg=\UTBVfin$. It is enough to construct the set \Eset. To this end we shall use transfinite induction on the ordinal \B (recall that every well-ordered set is isomorphic to an ordinal). If $v\in\B$ is the least element, then from \Cref{lem1} we obtain an idempotent $e_v$ with $\Supp_{\B}(e_v)=\{v\}$. Define $\Eset_{\{v\}}=\{e_v\}$.
Let $\alpha\leq\B$ be an ordinal, and assume that $\Eset_\gamma\subset\Aalg_1$ has been constructed for every ordinal $\gamma<\alpha$ and consists of pairwise orthogonal idempotents whose supports are singletons and satisfy \begin{equation}\label{eq:summable}\bigoplus_{e\in\Eset_\gamma}e(\V)\oplus\bigcap_{e\in\Eset_\gamma}\ker e=\V.\end{equation}

If $\alpha$ is not a limit ordinal, define $\Eset'=\Eset_{\alpha-1}$. Otherwise, $\alpha$ is a limit ordinal and, in this case, let $S=\bigcup_{\gamma<\alpha}\gamma\subset\B$. It follows from \Cref{lem:idemp:keep:grad} that there exists an idempotent $f\in\UTBV$ such that $S=\Supp_{\B}(f)$ and $f\Aalg f\subseteq \Aalg$ is still graded. Define $\Eset'=\bigcup_{\gamma<\alpha}f\Eset_\gamma f$, so that $\ker f\subseteq \ker e$ for all $e\in\Eset '$. 

We claim that the elements of $\Eset'$ are nonzero orthogonal idempotents whose supports are singletons. In effect, for all $e\in \bigcup_{\gamma<\alpha}\Eset_\gamma$, we have $\Supp_{\B}(f e f)=\Supp_{\B}(e)=\{u\}$ for some $u\in S$, therefore $f e f\neq 0$.
It follows from \Cref{cor:corner:UTBV} that $f\Aalg f\subset f(\UTBV) f$ is isomorphic to a subalgebra of $\UT_S V_S$, where $V_S=\Span S$ (see also \Cref{cor:corner:A}). The image of $f$ is the identity of $\UT_S V_S$ hence, under the isomorphism, we have
\[(f e_1 f)(f e_2 f)=fe_1e_2f=\begin{cases}f e_1 f,& e_1=e_2\\0,&e_1\neq e_2
\end{cases}\]
for all $e_1,e_2\in\bigcup_{\gamma<\alpha}\Eset_\gamma$. 

Further, due to the triangularity, we have $\bigoplus_{e\in\Eset'}e(\V)=\Span S=f(\V)$, hence
\[\bigoplus_{e\in\Eset'}e(\V)\oplus\bigcap_{e\in\Eset'}\ker e\supseteq  f(\V)\oplus\ker f=\V.\] In any case,  $\Eset'$ consists of pairwise orgthogonal idempotents whose supports are singletons and satisfies \cref{eq:summable}.

Next, let $\Balg=\bigcap_{e\in\Eset'}\{a-ae-ea+eae\mid a\in\Aalg\}$. This is a graded subalgebra of \Aalg (each $e\in\Aalg_1$ by construction). Here we notice that, since $\Eset'$ is summable to the idempotent $f\in\UTBV$ (indeed, the sum is an idempotent commuting with and having the same support as $f$; cf. \Cref{lem:iso:idemps}), then $\Balg$ is just $(1-f)\Aalg(1-f)$, albeit neither $f$ nor $1-f$ need to be in \Aalg. In particular, the supports of the elements in $\Balg$ are disjoint from $S$ and their union complements $S$. It follows from \Cref{cor:corner:A} that $\Balg$ is isomporphic to a subalgebra of $\UTfin_{\B'}\V'$ containing $\UT_{\to\B'}\V'$, where $\B'=\B\setminus S$ and $\V'=\Span\B'$.

If $v$ is the least element of $\B'$, then $\alpha=S\cup\{v\}$, and it follows from \Cref{lem1} that there exists a primitive idempotent $e_v\in\Balg_1$ with $\Supp_{\B'}(e_{v})=\{v\}$ (= $\Supp_{\B}(e_v)$). By construction, $e_v$ is orthogonal to every element of $\Eset'$. Define $\Eset_{\alpha}=\Eset'\cup\{e_v\}$.
It is clear that, by transfinite induction, the above process will eventually exhaust all of \B and that $\Eset=\Eset_\B$ fulfills all the requirements, including \cref{eq:summable}. Finally, since each $v\in\B$ occurs in the support of some element in \Eset, it follows that this set is complete.
\end{proof}

The following example shows that the condition given by \cref{eq:summable} is essential.

\begin{Example}
Let $(\B,\leq)$ be a well ordered basis of \V indexed by the set $\Ialg=\NN\cup\{\NN\}$. Suppose that $\Aalg=\UTBVfin$ is graded by a group and the grading is elementary, with $\deg e_{i,\NN}=1$ for all $i\in\Ialg$. Then the set
\[\Eset=\{e_i:=e_{ii}+e_{i,\NN}\mid i\in\NN \}\]
consists of primitive pairwise orthogonal idempotents in $\Aalg_1$, but it is not complete, since $v_{\NN}\notin\bigcup_{i\in\NN}\Supp_{\B}(e_i)$. It is clear that $v_{\NN}\notin\ker e_i$ for infinitely many $i\in\NN$, hence \cref{eq:summable} does not hold (cf. \cite[Lemma 2.11]{IMesR}). Further, the set cannot be completed in such a way as to remain orthogonal, since no idempotent in \Aalg with support $\{v_{\NN}\}$ can be orthogonal to all the elements in \Eset.
\end{Example}

In an analogous way as before, we obtain the following:

\begin{Cor}\label{cor:JR:graded:2}
Let $(\B,\leq)$ be a well-ordered basis of \V and assume that \UTBVfin is graded by a group \G. Let \Aalg be a graded subalgebra with $\UTBVlim\subseteq\Aalg\subseteq\UTBVfin$. Then the Jacobson Radical $J(\Aalg)$ is a graded ideal of \Aalg.\qed
\end{Cor} 

\subsection{Isomorphism classes of group gradings}\label{subsec:iso:classes}

In this subsection we shall derive a property concerning an automorphism of \UTBVlim. As a consequence, we will be able to classify the isomorphism classes of elementary gradings on it.

\begin{Lemma}\label{auto_prop}
Let $\varphi$ be an automorphism of \Aalg, where $\UTBVlim\subseteq\Aalg\subseteq\UTB$. Then, for each $i\le j$,
$$
e_{ii}\varphi(e_{ij})e_{jj}=\lambda e_{ij},\quad\lambda\ne0.
$$
\end{Lemma}
\begin{proof}
Note that $\{\varphi(e_{ii})\mid i\in\B\}$ is a set of pairwise orthogonal primitive idempotents. Thus, $\varphi$ induces a map $\alpha\colon\B\to\B$ such that $\alpha(i)$ is the unique element in $\mathrm{Supp}(\varphi(e_{ii}))$. Now, given $i$, $j\in\B$,
$$
0\ne\varphi(\F e_{ij})=\varphi(e_{ii}(\UTBVlim)e_{jj})=\varphi(e_{ii})(\UTBVlim)\varphi(e_{jj}).
$$
Thus, $i\le j$ if and only if $\alpha(i)\le\alpha(j)$. Since $\varphi^{-1}$ induces the map $\alpha^{-1}$, we see that $\alpha$ is the identity.
\end{proof}

As a consequence, we obtain the following:
\begin{Lemma}\label{gradpreserving}
Let $\Gamma$ and $\Gamma'$ be finite \G-graded good gradings on \UTBVlim, and $\overline{\Gamma}$ and $\overline{\Gamma'}$ the respective extensions to \Aalg, where $\UTBlim\subseteq\Aalg\subseteq\UTB$. If $\overline{\Gamma}\cong\overline{\Gamma'}$ then $\Gamma=\Gamma'$.
\end{Lemma}
\begin{proof}
Let $\psi\colon(\UTB,\overline{\Gamma})\to(\UTB,\overline{\Gamma'})$ be a \G-graded isomorphism. Then, for any $(x,y)\in I$, $e_{xx}$ and $e_{yy}$ are homogeneous of \G-degree $1$, with respect to both $\overline{\Gamma}$ and $\overline{\Gamma'}$. By \Cref{auto_prop}, $e_{xx}\psi(e_{xy})e_{yy}\ne0$. Thus,
$$
\deg_\Gamma e_{xy}=\deg_{\Gamma'}e_{xx}\psi(e_{xy})e_{yy}=\deg_{\Gamma'}e_{xy}.
$$
\end{proof}

\subsection{Conclusion}\label{subsec:conclusion}

We summarize our main results together, to obtain a complete classification of group gradings on the algebras \UTBVlim, \UTBVfin and \UTBV, where \V admits a well-ordered basis.
\begin{Thm}\label{statementmainthm}
Let \B be a well-ordered basis of a vector space \V and \G be an arbitrary group. Then any \G-grading on \UTBV and on \UTBVfin has finite support, is isomorphic to an elementary one, and is induced from a finite elementary grading on \UTBVlim.

Furthermore, let $\Gamma(\eta)$ and $\Gamma(\eta')$ be finite \G-gradings on \UTBVlim, and denote by $\overline{\Gamma(\eta)}$ and $\overline{\Gamma(\eta')}$ the respective induced gradings on \UTBV, and by $\widetilde{\Gamma(\eta)}$ and $\widetilde{\Gamma(\eta')}$ the induced on \UTBVfin. Then, the following statements are equivalent:
\begin{enumerate}
\item $\overline{\Gamma(\eta)}\cong\overline{\Gamma(\eta')}$,
\item $\widetilde{\Gamma(\eta)}\cong\widetilde{\Gamma(\eta')}$,
\item $\Gamma(\eta)=\Gamma(\eta')$,
\item $\eta=g\eta'$, for some $g\in \G$.
\end{enumerate}
\end{Thm}
\begin{proof}
For the first part, \Cref{thm:grading:class:utbv} and \Cref{thm:grading:class:utbvfin} tell us that every group grading on \UTBV and on \UTBVfin is isomorphic to a good grading. \Cref{lemma:goodgradingfinite} proves that such gradings are finite. Finally, \Cref{goodclosedinduced} concludes that such gradings are induced from finite good gradings on \UTBVlim.

Now, we prove the second part of the statement. $(i)\Rightarrow(iii)$ and $(ii)\Rightarrow(iii)$ are the content of \Cref{gradpreserving}. The implication $(iii)\Rightarrow(iv)$ is proved in \Cref{isomorphicelementarysequence}. The last ones $(iv)\Rightarrow(i)$ and $(iv)\Rightarrow(ii)$ are clear.
\end{proof}

Group gradings on \UTBVlim do not need to be finite. More precisely, we have:

\begin{Thm}\label{maincor}
Let \G be an arbitrary group, and \V be a \F-vector space with a well-ordered basis \B. Then, any \G-grading on $\mathrm{UT}$ is isomorphic to an elementary one, defined by a map $\gamma\colon\B\to \G$. Moreover, $\Gamma(\gamma)\cong\Gamma(\gamma')$ if and only if $\gamma=g\gamma'$, for some $g\in \G$.\qed
\end{Thm}

Recall that given two group gradings $\Gamma\colon\Aalg=\bigoplus_{g\in \G}\Aalg_g$ and $\Gamma'\colon\Aalg=\bigoplus_{h\in H}\Aalg'_h$, we say that $\Gamma'$ is a \highlight{refinement} of $\Gamma$ (or $\Gamma$ is a \highlight{coarsening} of $\Gamma'$) if for each $h\in H$, there exists $g\in \G$ such that $\Aalg_h'\subseteq\Aalg_g$. A grading that does not admit a proper refinement is said to be \highlight{fine}. The classification of fine group gradings up to equivalence on an finite-dimensional algebra is an important topic, and it is equivalent (under certain circumstances) to the classification of maximal tori on the algebraic group of automorphisms of the given algebra.

We shall discuss fine group gradings on the algebra $\mathrm{UT}=\bigcup_{n\in\NN}\mathrm{UT}_n(\F)$. It corresponds to the case where \B is indexed by \NN. From \Cref{sec_elem_grad}, we construct a $\ZZ^\NN$-grading $\Gamma_U$ on $\mathrm{UT}$ via $\gamma\colon\NN\to\ZZ^\NN$, where $\gamma(i)\colon\NN\to\ZZ$ is such that $\gamma(i)(j)=\delta_{ij}$. All the homogeneous components, but the identity one, are either $0$ or 1-dimensional, and spanned by some matrix unit. On the other hand, every elementary grading on $\mathrm{UT}$ is such that every homogeneous component is spanned by a set of matrix units. Thus, $\Gamma_U$ is a refinement of any elementary grading on $\mathrm{UT}$. As a consequence, $\Gamma_U$ is the \highlight{unique} fine grading on $\mathrm{UT}$ up to equivalence.

On the other hand, there is \textbf{no} fine grading on \UTBV. Indeed, from \Cref{statementmainthm}, any group grading on \UTBV has finite support, and, up to an isomorphism, is obtained from the extension of an elementary grading $\Gamma(\gamma)$ where $\gamma\colon\NN\to \G$. Now, we have
\begin{Lemma}\label{refinement}
If $\Gamma$ is a finite \G-grading on $\mathrm{UT}$, then it admits a proper finite refinement.
\end{Lemma}
\begin{proof}
From \Cref{maincor}, we may assume that $\Gamma\cong\Gamma(\gamma)$, where $\gamma\colon\NN\to \G$ satisfies $\gamma(1)=1$. Moreover, since the support of the grading is finite, we may assume that the image of $\gamma$ is finite, since $\deg e_{1i}=\gamma(i)$, $i\in\NN$. Now, let $H=G\times C_2$, where $C_2=\langle\alpha\mid\alpha^2=1\rangle$ is the group with 2 elements. We let $i$, $j\in\NN\setminus\{1\}$ be such that $i\ne j$ and $\gamma(i)=\gamma(j)$. Then, we define $\gamma'\colon\NN\to H$ via:
$$
\gamma'(\ell)=\left\{\begin{array}{ll} 
\alpha\gamma(i),&\text{ if $\ell=i$},\\
\gamma(\ell),&\text{ otherwise}.\end{array}\right.
$$
The image of $\gamma'$ is finite as well, so $\Gamma(\gamma')$ is a finite group grading. It does not coincide with $\Gamma(\gamma)$, since $\deg_{\Gamma(\gamma')}e_{1i}\ne\deg_{\Gamma(\gamma)}e_{1i}$. We claim that $\Gamma(\gamma')$ is a refinement of $\Gamma(\gamma)$. Indeed, the quotient $\varphi\colon H\to \G$ is a group homomorphism. Then, the group grading induced from $\varphi$ on $\Gamma(\gamma')$ is a coarsening of $\Gamma(\gamma')$. Moreover, it coincides with the the elementary grading obtained from the map $\varphi\circ\gamma'$. However, $\varphi\circ\gamma'=\gamma$. Hence, $\Gamma(\gamma')$ is a finite proper refinement of $\Gamma(\gamma)$.
\end{proof}

Thus, the restricted grading $\Gamma(\gamma)$ has a proper finite refinement $\Gamma(\gamma')$. Such grading has an extension $\overline{\Gamma(\gamma')}$ to a grading on \UTBV, which is a proper refinement of $\overline{\Gamma(\gamma)}$. Hence, every group grading on \UTBV has a proper refinement, and, in particular, it does not admit any fine group grading. We summarize the results on the following:
\begin{Prop}
Let $\B=\{v_i\mid i\in\NN\}$. Then
\begin{enumerate}
\setlength{\itemsep}{.3em}
\item there is no fine grading on \UTBV, and
\item there is a unique fine grading on \UTBVlim, up to an equivalence.
\end{enumerate}\qed
\end{Prop}

\section{Further comments: normed algebras\label{normed_space}}
It is interesting to investigate the notions of continuous and closed gradings in the context of normed and Banach algebras.

\begin{Lemma}[{\cite[Proposition 1.8.10]{Megg}}]
Let $X=Y_1\oplus\cdots\oplus Y_m$ be a Banach space given by the internal direct sum of the closed subspaces $Y_1$, \dots, $Y_m$. Then $X\cong Y_1\times\cdots\times Y_m$ as normed spaces.
\end{Lemma}

As a consequence, we obtain:
\begin{Lemma}
Let $X=Y_1\oplus\cdots\oplus Y_m$ be a Banach space, where $Y_1$, \dots, $Y_m$ are closed subspaces. Then the projections $\pi_i\colon X\to X$, with $\pi_i(X)=Y_i$, are continuous.\qed
\end{Lemma}
As a further consequence, we obtain the following:
\begin{Prop}
Let $X$ be a Banach algebra. Then every closed finite \G-grading on $X$ is continuous.\qed
\end{Prop}
Thus, for Banach algebras, every finite closed grading is continuous. However, for normed algebras, this is no longer true, as the following example shows.
\begin{Example}\label{exem:closed:not:continuous}
Let $c_0=\{(a_n)_{n\in\NN}\mid\lim a_n=0\}$ be the set of sequences in \RR converging to $0$, which is a Banach space. Let
$$
M=\{(a_n)_{n\in\NN}\mid a_{2n}=2na_{2n-1},\forall n\in\NN\},
$$
$$
N=\{(b_n)_{n\in\NN}\mid b_{2n+1}=0,\forall n\in\NN\}.
$$
Then, it is known (see, for instance, \cite[Exercise 1.84]{Megg}) that $M$ and $N$ are closed subspaces of $c_0$, $M\cap N=0$, $\V:=M\oplus N\ne c_0$ and $\overline{\V}=c_0$. Then, $\V=M\oplus N$ is a closed vector space grading. However, it is \textbf{not} continuous. Indeed, otherwise, there would have an extension to a vector space grading on $c_0$ via $c_0=\overline{M}\oplus\overline{N}=M\oplus N\ne c_0$, a contradiction.
\end{Example}

\section*{Acknowledgements}
We are thankful to V.~Morelli Cortes for the useful discussions and for providing us the facts given in \Cref{normed_space}.

\end{document}